\documentclass[11pt,letterpaper]{amsart}
\usepackage[T1]{fontenc}
\usepackage[textwidth=360pt,textheight=615pt,centering]{geometry}
\usepackage{amsmath,amssymb,mathtools}
\usepackage{esint}
\usepackage{array,enumerate}
\usepackage{graphicx}
\usepackage{accsupp}
\usepackage{xcolor}
\usepackage{tikz}
\usetikzlibrary{arrows.meta,positioning}
\usepackage[hidelinks]{hyperref}
\hypersetup{pdftitle={Sharp lower bounds for the Hilbert transform with matrix weights}}
\newtheorem{theorem}{Theorem}[section]
\newtheorem{lemma}[theorem]{Lemma}
\newtheorem{proposition}[theorem]{Proposition}
\newtheorem{corollary}[theorem]{Corollary}
\theoremstyle{definition}

\theoremstyle{remark}

\theoremstyle{definition}

\numberwithin{equation}{section}

\newcommand{\R}{\mathbb R}
\newcommand{\Z}{\mathbb Z}
\newcommand{\one}{\mathbf 1}
\newcommand{\op}{\mathrm{op}}
\newcommand{\asympp}{\asymp_p}

\title[On the matrix $A_p$ conjecture]{On the matrix $A_p$ conjecture}
\author{Yong Jiao}
\address{School of Mathematics and Statistics, Central South University, HNP-LAMA, Changsha 410083, China}
\email{jiaoyong@csu.edu.cn}
\thanks{This paper was supported by the National Natural Science Foundation of China (Nos.11971419, 12125109, W2411005, 12361131578, 12671164), the National Key R\&D Program of China (No.2023YFA1010800), the Provincial Natural Science Foundation of Hunan (Nos.2024JJ1010, 2025ZYJ002, 2024RC1016, 2025JJ40007).}

\author{Xingyan Quan}
\address{School of Mathematics and Statistics, Central South University, HNP-LAMA, Changsha 410083, China}
\email{quanxingyan@csu.edu.cn}

\author{Lian Wu}
\address{School of Mathematics and Statistics, Central South University, HNP-LAMA, Changsha 410083, China}
\email{wulian@csu.edu.cn}

\author{Guangheng Xie}
\address{School of Mathematics and Statistics, Central South University, HNP-LAMA, Changsha 410083, China}
\email{xieguangheng@csu.edu.cn}

\date{}
\subjclass[2020]{Primary 42B20, Secondary 42B35, 47A30}
\keywords{Matrix weights, Hilbert transform, sharp weighted inequalities,
Haar shifts, martingales}

\begin{document}
\begin{abstract}
This paper provides a solution to the matrix $A_p$ conjecture.
For every \(1<p<\infty\), we construct a family of \(2\times2\)
matrix \(A_p\) weights $W$ with arbitrarily large characteristics for which
\[
\|H\|_{L^p(W)\to L^p(W)}
\ge c_p[W]_{A_p}^{1+1/[p(p-1)]},
\]
where \(H\) is the Hilbert transform and \(c_p>0\) depends only on \(p\).
\end{abstract}
\maketitle

\medskip

\section{Introduction}
Weighted norm inequalities for the Hilbert transform occupy a central place in harmonic analysis. For a scalar weight \(w\) and \(1<p<\infty\), Hunt, Muckenhoupt, and Wheeden \cite{hunt-muckenhoupt-wheeden-1973} characterized the boundedness of the Hilbert transform \(H\) on \(L^p(w)\) by the Muckenhoupt condition \(w\in A_p\). The subsequent quantitative theory seeks to determine the optimal dependence of the operator norm on the characteristic \([w]_{A_p}\). For the Hilbert transform, the sharp scalar estimate is
\begin{equation}\label{eq:scalar-sharp-bound}
    \|H\|_{L^p(w)\to L^p(w)}
    \lesssim_p
    [w]_{A_p}^{\max\{1,\,1/(p-1)\}},
\end{equation}
and the exponent in \eqref{eq:scalar-sharp-bound} cannot be decreased; see \cite{petermichl-2007,hytonen-2012}.

The theory of matrix weights originated in questions concerning multivariate stationary processes, vector-valued Hardy spaces, and Toeplitz operators. Recall that, for a fixed dimension \(d\geq 2\), a matrix weight is a locally integrable function
$
W\colon\mathbb R\to\mathbb C^{d\times d}
$
whose values are positive definite almost everywhere. The weighted space \(L^p(W)\) is equipped with the norm
\[
\|f\|_{L^p(W)}
:=
\left(
\int_{\mathbb R}
\bigl|W(x)^{1/p}f(x)\bigr|^p\,dx
\right)^{1/p}.
\]
For \(p=2\), Treil and Volberg \cite{treil-volberg-1997} proved
the matrix version of the Hunt--Muckenhoupt--Wheeden theorem.
For the full range \(1<p<\infty\), Nazarov and Treil
\cite{nazarov-treil-1996} and Volberg \cite{volberg-1997}
characterized the boundedness of the Hilbert transform on
\(L^p(W)\) by the matrix \(A_p\) condition. The quantitative dependence of the operator norm on the matrix weight, however,
is fundamentally different from that in the scalar case, owing to the noncommutativity of the matrices \(W(x)\) and \(W(y)\).

Cruz-Uribe, Isralowitz, and Moen
\cite{cruzuribe-isralowitz-moen-2018} obtained a quantitative upper
bound for the full range \(1<p<\infty\). To state their result,
we use Roudenko's formulation of the matrix \(A_p\) condition
\cite{roudenko-2003}, which is equivalent to the original one
introduced in \cite{nazarov-treil-1996,volberg-1997} and is also used in
\cite{cruzuribe-isralowitz-moen-2018}. The corresponding
characteristic is
\[
    [W]_{A_p}
    :=
    \sup_I
    \frac{1}{|I|}\int_I
    \left(
        \frac{1}{|I|}\int_I
        \|W(x)^{1/p}W(y)^{-1/p}\|_{\mathrm{op}}^q\,dy
    \right)^{p-1}dx,
\]
where \(q=p/(p-1)\) is the conjugate exponent of \(p\) and the
supremum is taken over all finite intervals \(I\subset\mathbb R\).
For scalar weights, this definition reduces to the usual
\(A_p\) characteristic. We say that \(W\in A_p\) if
\([W]_{A_p}<\infty\). By
\cite[Corollary~1.16]{cruzuribe-isralowitz-moen-2018}, every
Calder\'on--Zygmund operator \(T\) and every matrix \(A_p\) weight
\(W\) satisfy
\begin{equation}\label{eq:known-matrix-upper-bound}
    \|T\|_{L^p(W)\to L^p(W)}
    \lesssim_{p,d,T}
    [W]_{A_p}^{\,1+\frac{1}{p(p-1)}}.
\end{equation}
For \(p=2\), this reduces to the \([W]_{A_2}^{3/2}\) estimate obtained
by Nazarov, Petermichl, Treil, and Volberg
\cite{nazarov-petermichl-treil-volberg-2017} through convex-body
domination. The exponent in \eqref{eq:known-matrix-upper-bound}
was not initially expected to be optimal. Indeed, immediately after
establishing this estimate, Cruz-Uribe, Isralowitz, and Moen
conjectured that the sharp exponent should coincide with the scalar
exponent \(\max\{1,1/(p-1)\}\); see
\cite[Remark~1.17]{cruzuribe-isralowitz-moen-2018}.

The first decisive obstruction to this conjectural picture arose at \(p=2\). Domelevo, Petermichl, Treil, and Volberg \cite{domelevo-petermichl-treil-volberg-2024} constructed matrix \(A_2\) weights for which
\[
\|H\|_{L^2(W)\to L^2(W)}
\gtrsim
[W]_{A_2}^{3/2},
\]
thereby disproving the matrix \(A_2\) conjecture and showing that the previously known upper bound is sharp at \(p=2\). The sharp exponent for \(p\neq2\), however, remained open. In particular, after recording \eqref{eq:known-matrix-upper-bound} and the sharp result at \(p=2\), Cruz-Uribe explicitly listed the determination of the sharp exponents for \(p\neq2\) as an open problem in \cite{cruzuribe-2025}. See also \cite{lerner-li-ombrosi-riverarios-2024} for an earlier explicit formulation asking whether the exponent in \eqref{eq:known-matrix-upper-bound} could be improved to the scalar one. The purpose of this paper is to resolve this problem. We show that the exponent in \eqref{eq:known-matrix-upper-bound} is sharp for every \(1<p<\infty\). Our main result is the following.

\begin{theorem}\label{thm:main-1}
    Let \(1<p<\infty\). There exist constants \(c_p,C_p>0\), depending only on \(p\), such that, for every sufficiently large \(Q\), there exist a \(2\times2\) matrix weight \(W\) satisfying
$[W]_{A_p}\leq C_pQ$
    and a nonzero function \(f\in L^p(W)\) such that
    \[
        \|Hf\|_{L^p(W)}
        \geq
        c_p Q^{1+1/[p(p-1)]}\|f\|_{L^p(W)}.
    \]

\end{theorem}

Together with the upper bound \eqref{eq:known-matrix-upper-bound}, Theorem~\ref{thm:main-1} determines the sharp exponent for the Hilbert transform throughout the full range \(1<p<\infty\). Compared with the scalar theory, the matrix exponent is larger by \(1/[p(p-1)]\) when \(p>2\), and by \((p-1)/p\) when \(1<p<2\). Thus the conjectured scalar exponent fails for every \(p\neq2\) as well.

The quantitative extrapolation theorem of Bownik and Cruz-Uribe
\cite{bownik-cruzuribe-2026}, applied to the sharp \(L^2\) upper bound,
gives the exponent
\(\frac32\max\{1,1/(p-1)\}\).
This exponent is strictly larger than that in
\eqref{eq:known-matrix-upper-bound} whenever \(p\ne2\), so this
extrapolated upper bound does not determine the sharp exponent
away from \(p=2\). We instead prove
Theorem~\ref{thm:main-1}  by a direct construction. We first establish the corresponding lower bound
for the odd Haar-shift difference \(\mathbb S-\mathbb S^*\) and
then transfer it to the Hilbert transform on the real line.

For $p\ge2$, we construct finite dyadic matrix weights using
scalar averaging identities adapted to the powers $p$ and
$q=p/(p-1)$.
The construction alternates rotations and stretches, drawing
on ideas from the $p=2$ work of Domelevo, Petermichl, Treil,
and Volberg
\cite{domelevo-petermichl-treil-volberg-2024}. Their construction
prescribes the matrix averages of the weight and its inverse.
In our construction, the assigned matrices need not be the averages
of the final weight. We keep the eigenvalues unchanged
during rotations and choose the stretching parameters separately
in the two eigendirections. At each stretching step, we impose
these scalar averaging identities.
Scalar moment estimates,
together with control of the accumulated rotations, then give
the required matrix $A_p$ bound.

We next construct a finite martingale test function
${\bf f}_N$ whose average on each continuing interval lies in the
eigendirection corresponding to the smaller eigenvalue of the
assigned matrix.
Its weighted norm can be computed from the scalar values on the
stopping intervals. 
We also estimate the component of $(\mathbb S-\mathbb S^*){\bf f}_N$
along the orthogonal eigendirection on each relevant stopping interval.
This component has a fixed sign, which controls the
cancellation between the shift and its adjoint.
For sufficiently large \(N\), the weighted norm estimate for \(\mathbf f_N\) and 
the signed pointwise estimate for \((\mathbb S-\mathbb S^*)\mathbf f_N\) yield
\[
\|{\bf f}_N\|_{L^p(W_N)}^p
\lesssim_p x_{a,0}Q^{-1/(p-1)},
\qquad
\|(\mathbb S-\mathbb S^*){\bf f}_N\|_{L^p(W_N)}^p
\gtrsim_p x_{a,0}Q^p,
\]
where $x_{a,0}$ is the smaller eigenvalue of the initial matrix.
Taking the ratio and then the $p$-th root gives the exponent
$1+1/[p(p-1)]$. We also obtain an explicit normalized pair of
test functions for the corresponding integral lower bound.
When $p=2$, the two stretching parameters can be chosen equal
and the same proof applies. For $1<p<2$, we exchange the two
test functions from the construction at the conjugate exponent
and estimate the $A_p$ characteristic of the resulting weight
directly.

For the passage to the Hilbert transform, we develop a
distributional argument that controls the matrix $A_p$
characteristic for every $1<p<\infty$.
At each stage, we compare the relevant distributions with
those of the original dyadic weight.
This comparison prevents the bounds from accumulating
over successive remodeling steps and yields uniform control
on arbitrary intervals.
Together with weighted norm convergence, it allows the
dyadic lower bound to pass to the Hilbert transform. We should mention that
the passage uses the transfer identity in
\cite{domelevo-petermichl-treil-volberg-2024}
and the remodeling framework in
\cite{kakaroumpas-treil-2021}.

The remainder of the paper is organized as follows. In Section~\ref{sec:2}, we construct the dyadic matrix weights through alternating rotations and stretches and prove an \(A_p\) bound uniform in the number of generations. Section~\ref{sec:test-function} consists of the  construction of the martingale test functions and the  estimates of their weighted norms. In Section~\ref{sec:haar-lower}, we establish the lower bound for the odd Haar shift difference \(\mathbb S-\mathbb S^*\). Finally, Section~\ref{sec:hilbert-lower} transfers the dyadic lower bounds to the Hilbert transform on the real line by means of quasi-periodization and remodeling, while preserving the matrix \(A_p\) characteristic up to a multiplicative constant.

\section{Construction of the dyadic \texorpdfstring{$A_p$}{Ap} weight}\label{sec:2}

Throughout this section, assume that $2\le p<\infty$. 	
Let $N\ge2$ be an integer that will be chosen sufficiently large later. 
We construct finite weights $W_N=W_{N,Q,\delta_0}$, depending on $Q\gg1$ and 
sufficiently small $\delta_0>0$, such that $[W_N]_{A_p}^{\mathrm{dy}}\lesssim_p Q$ 
uniformly in $N$ (Proposition~\ref{lem:dyadic-Ap-upper}). We suppress the indices $Q$ and $\delta_0$. 

We use half-open dyadic intervals. The symbol $\one_I$
denotes the function equal to one on $I$ and zero elsewhere.
For $m\ge0$, let 
\begin{equation*}
\begin{aligned}
\mathcal D_m
&=\left\{
[j2^{-m},(j+1)2^{-m})\mid 0\le j<2^m
\right\}.
\end{aligned}
\end{equation*}
Set $\mathcal D
=\bigcup_{m\ge0}\mathcal D_m.$ Let $I_0=[0,1).$
For $I\in\mathcal D\setminus\{I_0\}$, let $\tilde I$ denote its parent. For every $I\in\mathcal D$, let $I_+$ and $I_-$ denote its left and right children, respectively.

To obtain the sharp lower bound, the rotations and stretches must
accumulate over many generations. We arrange mean identities for
the scalar parameters of the continuing construction. After
truncation, the primal scalar averages remain exact, while the
actual dual scalar averages satisfy the upper bounds needed for
the $A_p$ estimate.
Set $\mathcal P_0=\{I_0\}$. The construction continues on the
intervals in $\mathcal P_i$ and stops on the intervals in $E_i$.
Define these families by

\begin{equation}\label{constru-P}
\mathcal P_{i}=\begin{cases}
\bigcup_{I\in\mathcal P_{i-2}}\{I_{++},I_{--}\}
,&i=2n+2,\, 0\leq n<N-2,\\

\bigcup_{I\in\mathcal P_{i-1}}\{I_+,I_-\}
,&i=2n+1,\, 0\leq n<N-1.\end{cases}\subseteq\mathcal D_{i},
\end{equation}
and
\begin{equation}\label{construc-E}
E_{i}=\begin{cases}\bigcup_{I\in\mathcal P_{i-2}}\{I_{+-},I_{-+}\}
,&i=2n+2,\,0\leq n<N-1,\\
\bigcup_{I\in\mathcal P_{i-4}}\{I_{++},I_{--}\}
,& i=2N.
\end{cases}
\end{equation}
Thus $E_{2n+2}\subseteq\mathcal D_{2n+2}$ for $0\le n<N-1$, whereas
$E_{2N}\subseteq\mathcal D_{2N-2}$. The index $2N$ names the last
family. Its dyadic generation is $2N-2$.
Set
\begingroup
$$\mathcal P
=\bigcup_{n=0}^{N-2}\mathcal P_{2n},\qquad 	E
=\bigcup_{n=1}^{N}E_{2n}.$$
\endgroup
By construction, we find  the following properties.
\begin{proposition}\label{pro-1}

\begin{enumerate}[\rm (i)]
\item The sets in $E$ are disjoint and $\bigcup_{I\in E}I=I_0$.
\item For every $0\le n\le N-2$ and
$J\in\mathcal P_{2n}\cup\mathcal P_{2n+1}$, the intervals
$I_\alpha\subseteq J$ belonging to
$\bigcup_{n<m\le N}E_{2m}$ partition $J$, and
$$\Bigl|\bigcup_{\substack{I_\alpha\in E_{2m}\\I_\alpha\subseteq J}}I_\alpha\Bigr|
=2^{-(m-n)}|J|\quad(n<m<N),$$
$$\Bigl|\bigcup_{\substack{I_\alpha\in E_{2N}\\I_\alpha\subseteq J}}I_\alpha\Bigr|
=2^{-(N-n-1)}|J|.$$

\item Every $I\in\mathcal D$ is either contained in an interval
$J\in E$, or belongs to $\mathcal P_{2n}\cup\mathcal P_{2n+1}$ for
some $0\le n\le N-2$.
\end{enumerate}

\end{proposition}

\begin{proof}
	At each step before termination, an even-generation interval
	$K$ is divided into four grandchildren.
	The inner grandchildren $K_{+-}$ and $K_{-+}$ enter the
	stopping family, while the outer grandchildren $K_{++}$ and
	$K_{--}$ continue the construction.
	Each pair occupies half of $K$.
	For an odd-generation interval $K_+$ or $K_-$, exactly one
	child stops and the other continues, again with equal lengths.
	At the final step, the remaining outer grandchildren enter
	$E_{2N}$.
	
	Fix $J\in\mathcal P_{2n}\cup\mathcal P_{2n+1}$.
	At each subsequent step, half of the remaining part of $J$
	enters the next stopping family.
	Consequently,
	\[
	\big|
	\bigcup_{\substack{I\in E_{2m}\\ I\subseteq J}}I
	\big|
	=2^{-(m-n)}|J|,
	\qquad n<m<N.
	\]
	After these steps, the remaining part has measure
	$2^{-(N-n-1)}|J|$ and consists precisely of the intervals
	in $E_{2N}$ contained in $J$.
	The stopping intervals are pairwise disjoint, and these
	successive decompositions exhaust $J$.
	This proves (ii), and taking $J=I_0$ gives (i).
	
	Finally, follow the dyadic ancestors of any $I\in\mathcal D$.
	Either a stopping interval is reached before or at $I$,
	in which case $I$ is contained in a member of $E$,
	or $I$ belongs to one of the continuing families.
	Every branch terminates by generation $2N-2$.
	This proves (iii).
\end{proof}

\subsection{Starting point}\label{sec:starting-point}
Let ${\bf a}_{I_0}, {\bf b}_{I_0}$  form an orthonormal basis of $\mathbb R^2$. 
All vectors are regarded as column vectors in $\R^2$. 
Thus, ${\bf a}_{I_0}{\bf a}_{I_0}^*$ and ${\bf b}_{I_0}{\bf b}_{I_0}^*$ respectively denote the rank-one
orthogonal projections onto ${\rm span}\{{\bf a}_{I_0}\}$ and ${\rm span}\{{\bf b}_{I_0}\}$.
Given $0<\delta_0<1$, set

\begin{equation*}
x_{a,0}=\delta_0^p,\qquad   x_{b,0}=1.
\end{equation*}
On the initial interval $I_0$, we choose $X_{I_0}$ to be diagonal with respect to the orthonormal basis $\{{\bf a}_{I_0},{\bf b}_{I_0}\}$, so that
$$X_{I_0}=x_{a,0}{\bf a}_{I_0}{\bf a}_{I_0}^*+x_{b,0}{\bf b}_{I_0}{\bf b}_{I_0}^*.$$

\subsection{Rotation}\label{rotation}
Set the rotation angle in the first step to be $\theta_0=\arctan\delta_0$.
We rotate the eigenvectors ${\bf a}_{I_0}, {\bf b}_{I_0}$ through the angle $\theta_0$. Define
$$	{\bf a}_{I_0,\pm}=\cos\theta_0\,{\bf a}_{I_0}\pm\sin\theta_0\, {\bf b}_{I_0},
\qquad {\bf b}_{{I_0},\pm}=\mp\sin\theta_0\,{\bf a}_{I_0}+\cos\theta_0\, {\bf b}_{I_0}.$$
Recall that $I_{0,+}, I_{0,-}\in\mathcal P_1$ are the left and right children of $I_0$.
We keep the eigenvalues $x_{a,0},x_{b,0}$ and apply these rotations to $I_{0,+}, I_{0,-}$, respectively, and set
$$X_{I_{0,\pm}}=x_{a,0}{\bf a}_{I_{0,\pm}}{\bf a}_{I_{0,\pm}}^*+x_{b,{0}}{\bf b}_{I_{0,\pm}}{\bf b}_{I_{0,\pm}}^*.$$

This choice of $\theta_0$ controls the weighted size of the projection of
${\bf a}_{I_0}$ onto the new direction ${\bf b}_{I_{0,\pm}}$. Indeed,
\[
x_{b,0}^{1/p}
\bigl|\langle {\bf a}_{I_0},{\bf b}_{I_{0,\pm}}\rangle\bigr|
=x_{b,0}^{1/p}\sin\theta_0
=\frac{x_{a,0}^{1/p}}{\sqrt{1+\delta_0^2}}
\le x_{a,0}^{1/p}.
\]
Thus this projection, after weighting, is no larger than the original
weighted size of ${\bf a}_{I_0}$.

\subsection{Stretching}\label{sketching}

We choose the stretching parameters separately in the ${\bf a}$ and ${\bf b}$
directions. For either direction, let the current scalar values be
$x,y>0$, with
\[
xy^{p-1}=Q.
\]
On the two children of an interval, assign the values
\[
\begin{array}{c|cc}
	& \text{continuing child} & \text{stopping child} \\ \hline
	x\text{-value} & sx & (2-s)x \\
	y\text{-value} & ty & (2-t)y
\end{array}
\]
The children have equal length, so both scalar averages are preserved:
\[
\frac{sx+(2-s)x}{2}=x,
\qquad
\frac{ty+(2-t)y}{2}=y.
\]
We require the product $xy^{p-1}$ to remain $Q$ on the continuing
child and to equal $1$ on the stopping child. Thus
\[
(sx)(ty)^{p-1}=Q
\quad\Longleftrightarrow\quad
st^{p-1}=1,
\]
and
\[
((2-s)x)((2-t)y)^{p-1}=1
\quad\Longleftrightarrow\quad
Q(2-s)(2-t)^{p-1}=1.
\]
In the ${\bf a}$ direction, take
\[
(s,t)=\bigl(r_a^{-(p-1)},r_a\bigr);
\]
in the ${\bf b}$ direction, take
\[
(s,t)=\bigl(r_b,r_b^{-1/(p-1)}\bigr).
\]
In both cases $st^{p-1}=1$. The condition on the stopping child
therefore gives
\begin{align}
	Q\bigl(2-r_a^{-(p-1)}\bigr)(2-r_a)^{p-1}=1, \quad
	Q(2-r_b)\bigl(2-r_b^{-1/(p-1)}\bigr)^{p-1}=1,
	\label{eq:rb}
\end{align}
where the stretching parameters $r_a$, $r_b\in(1,2)$.
Such parameters exist for every $Q>1$ because each left-hand side of \eqref{eq:rb} is continuous on $[1,2]$, equals $Q$ at $1$, and equals $0$ at $2$.
Set
\begin{equation*}
\ell=(r_a^{p-1}r_b)^{1/p}.
\end{equation*}
When $p=2$, the two equations in \eqref{eq:rb} coincide.
We choose $r_a=r_b=r$, so that $\ell=r$.
We write $A\asymp_p B$ when both $A\lesssim_p B$ and $B\lesssim_p A$. The following asymptotic properties of the stretch parameters will be used in the sequel.

\begin{lemma}
\label{lem:stretch-parameter-asymptotics}
Let $2\le p<\infty$, let $Q$ be sufficiently large, and let
$r_a,r_b\in(1,2)$ satisfy
\eqref{eq:rb}.  For $\ell$ defined above, we have
\begin{equation}
2-r_a\asympp Q^{-1/(p-1)}\qquad
2-r_b\asympp Q^{-1},
\label{eq:d-tau-asymptotics}
\end{equation}
and
\begin{equation}
2-\ell\asympp Q^{-1/(p-1)}.
\label{eq:d-tau-asymptotics2}
\end{equation}
\end{lemma}

\begin{proof}

The estimate in \eqref{eq:d-tau-asymptotics} follows directly from \eqref{eq:rb}, since $2-r_a^{-(p-1)},2-r_b^{-1/(p-1)}\in(1,2)$. To prove \eqref{eq:d-tau-asymptotics2}, use the identity $1-xy=1-x+x(1-y)$ and the estimate $1-(1-t)^\alpha\asymp_\alpha t$ for $\alpha>0$ and $0\le t\le1/2$, which follows from the mean value theorem.
Then we have
\begin{align*}
2-\ell&=2-(r_a^{p-1}r_b)^{1/p}=2(1-(r_a/2)^{\frac{p-1}{p}}(r_b/2)^{\frac{1}{p}})\\&=2\Big(1-({r_a}/{2})^{\frac{p-1}{p}}+(r_a/2)^{\frac{p-1}{p}} \Big(1-(r_b/2)^{\frac{1}{p}}\Big)\Big)\\&\asympp 1-({r_a}/{2})^{\frac{p-1}{p}}+1-({r_b}/{2})^{\frac1p}\\
&\asympp 2-r_a+2-r_b.
\end{align*}
Since $p\ge2$ and $Q\ge1$, we have
$Q^{-1}\le Q^{-1/(p-1)}$. Hence
\[
2-\ell\asymp_p (2-r_a)+(2-r_b)
\asymp_p Q^{-1/(p-1)}.
\]
This proves the lemma.
\end{proof}

Henceforth, take $Q$ sufficiently large that $3/2\le\ell<2$.
Using the stretching parameters $r_a$ and $r_b$, we define
\[
\begin{aligned}
x_{a,1} &= r_a^{-(p-1)} x_{a,0},\qquad &
x_{a,1}^{\#} = (2 - r_a^{-(p-1)}) x_{a,0},\\
x_{b,1} &= r_b x_{b,0},\qquad &
x_{b,1}^{\#} = (2 - r_b)\, x_{b,0}.
\end{aligned}
\]
Recall that $\mathcal D_2 = \{I_{0,+,+}, I_{0,+,-}, I_{0,-,+}, I_{0,-,-}\}$.
When $N\ge3$, the construction in \eqref{constru-P} and \eqref{construc-E} gives
$
\mathcal P_2 = \{I_{0,+,+}, I_{0,-,-}\}$ and $
E_2 = \{I_{0,+,-}, I_{0,-,+}\}.$
For each $I \in \mathcal P_2$, we set
\[
X_{I} = x_{a,1} \mathbf{a}_I \mathbf{a}_I^* + x_{b,1} \mathbf{b}_I \mathbf{b}_I^*,
\]
where $\mathbf{a}_I = \mathbf{a}_{\tilde I}$ and $\mathbf{b}_I = \mathbf{b}_{\tilde I}$, and then proceed to the next step. For each interval $I \in E_2$, we similarly define
\[
X_{I}^\# = x_{a,1}^{\#} \mathbf{a}_I \mathbf{a}_I^* + x_{b,1}^{\#} \mathbf{b}_I \mathbf{b}_I^*,
\]
with $\mathbf{a}_I = \mathbf{a}_{\tilde I}$ and
$\mathbf{b}_I = \mathbf{b}_{\tilde I}$. No further subdivision is
used on $I$. If $N=2$, the inner grandchildren form $E_2$, while the
outer grandchildren form $E_4$ and use $x_{a,1},x_{b,1}$.
Figure~\ref{fig:step} illustrates a step before the last step.
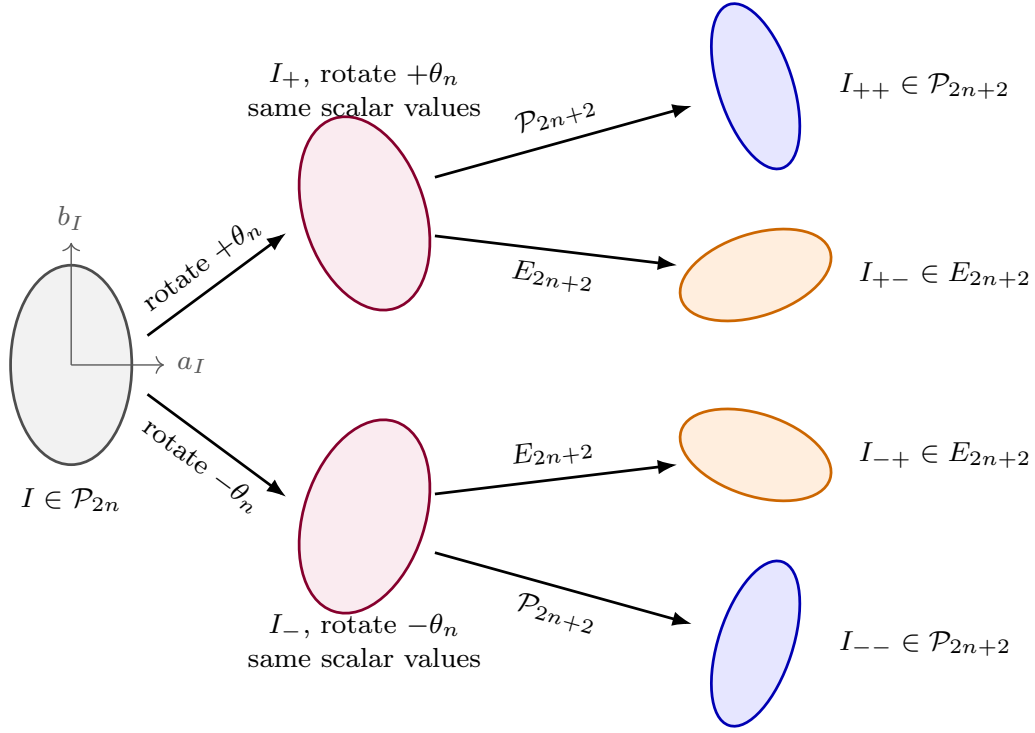
\begin{figure}[ht]
\centering
\BeginAccSupp{method=escape,ActualText={A branching diagram shows one step
of the dyadic construction. A parent interval has two children with
orthonormal bases rotated by opposite angles and unchanged scalar values.
Stretching on each child produces two grandchildren with the same basis
as that child. The two outer grandchildren continue the construction,
while the two inner grandchildren are stopping intervals.}}
\resizebox{0.9\textwidth}{!}{		\begin{tikzpicture}[			every node/.style={font=\scriptsize,align=center},			arr/.style={-{Latex[length=2mm]},thick},			parentellipse/.style={draw=black!70,fill=black!5,thick},			rotationellipse/.style={draw=purple!70!black,fill=purple!8,thick},			blueellipse/.style={draw=blue!70!black,fill=blue!10,thick},			orangeellipse/.style={draw=orange!80!black,fill=orange!13,thick}			]
\begin{scope}[shift={(0,0)}]				\draw[parentellipse] (0,0) ellipse [x radius=.62,y radius=1.02];				\draw[->,black!65] (0,0)--(.95,0) node[right] {$a_I$};				\draw[->,black!65] (0,0)--(0,1.25) node[above] {$b_I$};			\end{scope}			\node at (0,-1.38) {$I\in\mathcal P_{2n}$};

\begin{scope}[shift={(3,1.55)},rotate=18]				\draw[rotationellipse] (0,0) ellipse [x radius=.62,y radius=1.02];			\end{scope}			\node at (3,2.82) {$I_+$, rotate $+\theta_n$\\same scalar values};			\begin{scope}[shift={(3,-1.55)},rotate=-18]				\draw[rotationellipse] (0,0) ellipse [x radius=.62,y radius=1.02];			\end{scope}			\node at (3,-2.82) {$I_-$, rotate $-\theta_n$\\same scalar values};
\begin{scope}[shift={(7,2.85)},rotate=18]				\draw[blueellipse] (0,0) ellipse [x radius=.38,y radius=.88];			\end{scope}			\node[anchor=west] at (7.72,2.85)			{$I_{++}\in\mathcal P_{2n+2}$};			\begin{scope}[shift={(7,.92)},rotate=18]				\draw[orangeellipse] (0,0) ellipse [x radius=.80,y radius=.42];			\end{scope}			\node[anchor=west] at (7.92,.92)			{$I_{+-}\in E_{2n+2}$};			\begin{scope}[shift={(7,-.92)},rotate=-18]				\draw[orangeellipse] (0,0) ellipse [x radius=.80,y radius=.42];			\end{scope}			\node[anchor=west] at (7.92,-.92)			{$I_{-+}\in E_{2n+2}$};			\begin{scope}[shift={(7,-2.85)},rotate=-18]				\draw[blueellipse] (0,0) ellipse [x radius=.38,y radius=.88];			\end{scope}			\node[anchor=west] at (7.72,-2.85)			{$I_{--}\in\mathcal P_{2n+2}$};

\draw[arr] (.78,.30)--(2.20,1.35)
node[midway,above,sloped] {rotate $+\theta_n$};
\draw[arr] (.78,-.30)--(2.20,-1.35)
node[midway,below,sloped] {rotate $-\theta_n$};

\draw[arr] (3.72,1.92)--(6.35,2.65)
node[midway,above,sloped] {$\mathcal P_{2n+2}$};
\draw[arr] (3.72,1.32)--(6.20,1.02)
node[midway,below,sloped] {$E_{2n+2}$};
\draw[arr] (3.72,-1.32)--(6.20,-1.02)
node[midway,above,sloped] {$E_{2n+2}$};
\draw[arr] (3.72,-1.92)--(6.35,-2.65)
node[midway,below,sloped] {$\mathcal P_{2n+2}$};
\end{tikzpicture}
}
\EndAccSupp{}

\caption{{One step of rotation and stretching. The outer
grandchildren belong to $\mathcal P_{2n+2}$ and the inner
grandchildren belong to $E_{2n+2}$.}}
\label{fig:step}\end{figure}

\subsection{Iteration and scalar identities}
For $0\le m<N-1$, set
\begin{equation}\label{def-delta}
\delta_m=\delta_0\ell^{-m},\qquad \theta_m=\arctan\delta_m.
\end{equation}
We now iterate the rotation and stretching steps. Fix $0\le n<N-1$ and $I\in\mathcal P_{2n}$.
Assume that $$X_I=x_{a,n}{\bf a}_I{\bf a}_I^*+x_{b,n}{\bf b}_I{\bf b}_I^*.$$ Let $\Psi_I$ be the angle from the ordered orthonormal basis
$({\bf a}_{I_0},{\bf b}_{I_0})$ to $({\bf a}_I,{\bf b}_I)$ with
$\Psi_{I_0}=0$. Thus
\[\begin{aligned}
{\bf a}_I=\cos{\Psi_I}\,{\bf a}_{I_0}+\sin{\Psi_I}\,{\bf b}_{I_0},\qquad {\bf b}_I=-\sin{\Psi_I}\,{\bf a}_{I_0}+\cos{\Psi_I}\,{\bf b}_{I_0}.
\end{aligned}
\]

Apply the rotations from Subsection~\ref{rotation} to the children $I_\pm\in\mathcal P_{2n+1}$ to obtain
$$X_{I_\pm}=x_{a,n}{\bf a}_{I_\pm}{\bf a}_{I_\pm}^*+x_{b,n}{\bf b}_{I_\pm}{\bf b}_{I_\pm}^*,$$
where
\begin{align*}
{\bf a}_{I_\pm}&=\cos\theta_n\,{\bf a}_I\pm\sin\theta_n\,{\bf b}_I=\cos\Psi_{I_\pm}\,{\bf a}_{I_0}+\sin\Psi_{I_\pm}\,{\bf b}_{I_0},\\
{\bf b}_{I_\pm}&=\mp\sin\theta_n\,{\bf a}_I+\cos\theta_n\,{\bf b}_I=-\sin\Psi_{I_\pm}\,{\bf a}_{I_0}+\cos\Psi_{I_\pm}\,{\bf b}_{I_0},\\
\Psi_{I_\pm}&=\Psi_I\pm\theta_n.
\end{align*}

Next apply the stretching step. Every grandchild $J$ uses the same
ordered orthonormal basis as its parent. Hence
${\bf a}_J={\bf a}_{\tilde J}$, ${\bf b}_J={\bf b}_{\tilde J}$, and
$\Psi_J=\Psi_{\tilde J}$. If $0\le n<N-2$ and
$J\in\mathcal P_{2n+2}$, then we set
$$X_J=x_{a,n+1}{\bf a}_J{\bf a}_J^*+x_{b,n+1}{\bf b}_J{\bf b}_J^*.$$
If $0\le n<N-1$ and $J\in E_{2n+2}$, then we set
$$X_J^\#=x_{a,n+1}^\#{\bf a}_J{\bf a}_J^*+x_{b,n+1}^\#{\bf b}_J{\bf b}_J^*,$$
where the scalar values satisfy
\begin{align*}
&{x_{a,n+1}}=r_a^{-(p-1)}x_{a,n}, &{x_{a,n+1}^\#}&=\bigl(2-r_a^{-(p-1)}\bigr)x_{a,n},
\\
&{x_{b,n+1}}=r_b\,x_{b,n},&	{x_{b,n+1}^\#}&=(2-r_b)\,x_{b,n}.
\end{align*}

At the last step, the outer grandchildren are placed in $E_{2N}$, where the iteration ends.
For each $I\in E_{2N}$, set
\[
X_I^\#=x_{a,N-1}{\bf a}_I{\bf a}_I^*
+x_{b,N-1}{\bf b}_I{\bf b}_I^*.
\]
Every dyadic descendant of an interval in $E$ uses the same ordered
orthonormal basis and the same angle.

The scalar values and angles satisfy the following identities.
\begin{proposition}\label{W-pro}
\begin{enumerate}[\rm (i)]

\item For $0\le n\le N-1$, we have
\begin{equation*}
x_{a,n}=x_{a,0}r_a^{-(p-1)n},\qquad x_{b,n}=x_{b,0}r_b^n.
\end{equation*}
For $1\le n\le N-1$, the values used on $E_{2n}$ are
\begin{equation*}
x_{a,n}^\#=x_{a,0}(2-r_a^{-(p-1)})r_a^{-(p-1)(n-1)},\qquad
x_{b,n}^\#=x_{b,0}(2-r_b)r_b^{n-1}.
\end{equation*}

\item For $0\le m<N-1$,
$$\delta_m=\left(\frac{x_{a,m}}{x_{b,m}}\right)^{1/p}=r_b^{-m/p}r_a^{-(p-1)m/p}\delta_0=\ell^{-m}\delta_0.$$
\item
Let $I\in\mathcal P_{2n}$ with $0\le n\le N-2$, or $I\in E_{2n}$ with $1\le n\le N-1$. For $0\le m<n$, let $K_m\in\mathcal P_{2m}$ be the unique ancestor of $I$ at generation $2m$. Set $\sigma_m=1$ if $I\subseteq (K_m)_+$ and $\sigma_m=-1$ if $I\subseteq (K_m)_-$. Then
$$\Psi_I=\sum_{m=0}^{n-1}\sigma_m\theta_m,\qquad |\Psi_I|\le\sum_{m=0}^{n-1}\theta_m\lesssim_p\theta_0.$$
For $I\in E_{2N}$, the same formula holds with $n=N-1$, since these intervals have undergone $N-1$ rotations. For an odd-generation interval $I_\pm\in\mathcal P_{2n+1}$, the angle is $\Psi_{I_\pm}=\Psi_I\pm\theta_n$.
Moreover, if $I\in\mathcal P_{2n}\cup\mathcal P_{2n+1}$, $0\le n\le N-2$, and $J\in\mathcal D$ satisfies $J\subseteq I$, then
$$|\Psi_J-\Psi_I|\le\sum_{m=n}^{N-2}\theta_m\lesssim_p\theta_n\lesssim\delta_n.$$
If $I\in E$ and $J\subseteq I$, then $\Psi_J=\Psi_I$.

\item For $i\in\{a,b\}$, define $y_{i,n}=(Q/x_{i,n})^{1/(p-1)}$ for $0\le n\le N-1$, and $y_{i,n}^\#=(x_{i,n}^\#)^{-1/(p-1)}$ for $1\le n\le N-1$. Then, for $0\le n\le N-1$,

$$
\begin{aligned}
y_{a,n}=y_{a,0}r_a^{n},
\quad
y_{b,n}=y_{b,0}r_b^{-n/(p-1)}
\end{aligned}$$
and, for $1\le n\le N-1$,
\begin{equation*}
\begin{aligned}
y_{a,n}^\#&=(2-r_a)y_{a,0}r_a^{n-1},
\quad
y_{b,n}^\#=(2-r_b^{-1/(p-1)})y_{b,0}r_b^{-{(n-1)}/(p-1)}.
\end{aligned}
\end{equation*}

For $1\le n\le N-1$, the mean identities are
\begin{equation*}
\frac{x_{i,n}+x_{i,n}^{\#}}{2}=x_{i,n-1},
\qquad
\frac{y_{i,n}+y_{i,n}^{\#}}{2}=y_{i,n-1}.
\end{equation*}
The product identities are
\begin{equation*}
\begin{aligned}
x_{i,n}(y_{i,n})^{p-1}&=Q &&(0\le n\le N-1),\\
x_{i,n}^{\#}(y_{i,n}^{\#})^{p-1}&=1 &&(1\le n\le N-1).
\end{aligned}
\end{equation*}

\end{enumerate}
\end{proposition}
\begin{proof}
For (i), iterating the relations
\[
x_{a,n}=r_a^{-(p-1)}x_{a,n-1},
\qquad
x_{b,n}=r_bx_{b,n-1}
\]
gives the formulas for $x_{a,n}$ and $x_{b,n}$.
Substituting these into
\[
x_{a,n}^{\#}=(2-r_a^{-(p-1)})x_{a,n-1},
\qquad
x_{b,n}^{\#}=(2-r_b)x_{b,n-1}
\]
gives the remaining formulas in (i).

For (ii), the initial values and the definition of $\ell$ yield
\[
\left(\frac{x_{a,m}}{x_{b,m}}\right)^{1/p}
=\left(\frac{x_{a,0}}{x_{b,0}}\right)^{1/p}
(r_a^{p-1}r_b)^{-m/p}
=\delta_0\ell^{-m}
=\delta_m.
\]

For (iii), the rotation at step $m$ changes the angle by
$+\theta_m$ on the left child and by $-\theta_m$ on the right child. Stretching leaves the
ordered orthonormal basis unchanged. Starting from $\Psi_{I_0}=0$
therefore gives the stated formula for $\Psi_I$ and the relation
$\Psi_{I_\pm}=\Psi_I\pm\theta_n$. An interval in $E_{2N}$ has
undergone exactly $N-1$ rotations, so its angle is
$\sum_{m=0}^{N-2}\sigma_m\theta_m$.

Along the sequence of descendants from $I$ to $J$, each remaining
rotation changes the angle by at most $\theta_m$, and the angle
remains fixed once an interval in $E$ is reached. Thus, for
$I\in\mathcal P_{2n}\cup\mathcal P_{2n+1}$ and $J\subseteq I$,
\[
|\Psi_J-\Psi_I|\le\sum_{m=n}^{N-2}\theta_m.
\]
For $I\in\mathcal P_{2n+1}$, the rotation at step $n$ has already
occurred, so the same upper bound still applies. For sufficiently
large $Q$, Lemma~\ref{lem:stretch-parameter-asymptotics} gives
$\ell\ge3/2$. Since $0<\delta_m\le\delta_0<1$ and
$\theta_m=\arctan\delta_m$, we have
$\delta_m/2\le\theta_m\le\delta_m$. Consequently,
\[
\sum_{m=n}^{N-2}\theta_m
\le\delta_n\sum_{j=0}^{\infty}\ell^{-j}
=\frac{\ell}{\ell-1}\delta_n
\le3\delta_n
\le6\theta_n.
\]
Taking $n=0$ also gives the asserted bound for $|\Psi_I|$.
If $I\in E$, all its descendants retain its basis, so
$\Psi_J=\Psi_I$ for every $J\subseteq I$. This proves (iii).

For (iv), the definition of $y_{i,n}$ and (i) give
\[
y_{a,n}=r_a y_{a,n-1},
\qquad
y_{b,n}=r_b^{-1/(p-1)}y_{b,n-1}.
\]
Iterating proves the stated formulas for $y_{a,n}$ and $y_{b,n}$.
The parameter equations \eqref{eq:rb} imply
\[
\begin{aligned}
(2-r_a^{-(p-1)})^{-1/(p-1)}
&=Q^{1/(p-1)}(2-r_a),\\
(2-r_b)^{-1/(p-1)}
&=Q^{1/(p-1)}(2-r_b^{-1/(p-1)}).
\end{aligned}
\]
Using the definition of $y_{i,n}^{\#}$, we obtain
\[
y_{a,n}^{\#}=(2-r_a)y_{a,n-1},
\qquad
y_{b,n}^{\#}=(2-r_b^{-1/(p-1)})y_{b,n-1}.
\]
These relations give the formulas for the stopping values and,
together with the recurrences for $y_{i,n}$, their mean identities.
The mean identities for $x_{i,n}$ follow directly from the stretching
relations. Finally, both product identities follow from the
definitions of $y_{i,n}$ and $y_{i,n}^{\#}$.
\end{proof}

\subsection{The finite weight and its scalar averages}\label{sec:finite-truncation-eigenvalues}
The intervals in $E$ partition $I_0$. We define $W_N$ to equal the
positive definite matrix $X_I^\#$ on each $I\in E$. Thus
\begin{align*}
W_N&=\sum_{n=1}^{N-1}\sum_{I\in E_{2n}}\bigl(x_{a,n}^\#{\bf a}_I{\bf a}_I^*+x_{b,n}^\#{\bf b}_I{\bf b}_I^*\bigr)\one_I\\
&\quad+\sum_{J\in E_{2N}}\bigl(x_{a,N-1}{\bf a}_J{\bf a}_J^*+x_{b,N-1}{\bf b}_J{\bf b}_J^*\bigr)\one_J\\
&=\sum_{I\in E}X_I^\#\one_I.
\end{align*}
Consequently,
\begin{align*}
W_N^{-1/(p-1)}
&=\sum_{n=1}^{N-1}\sum_{I\in E_{2n}}\bigl(y_{a,n}^\#{\bf a}_I{\bf a}_I^*+y_{b,n}^\#{\bf b}_I{\bf b}_I^*\bigr)\one_I\\
&\quad+Q^{-1/(p-1)}\sum_{J\in E_{2N}}\bigl(y_{a,N-1}{\bf a}_J{\bf a}_J^*+y_{b,N-1}{\bf b}_J{\bf b}_J^*\bigr)\one_J.
\end{align*}
The factor $Q^{-1/(p-1)}$ on $E_{2N}$ follows from the definition
of $y_{i,N-1}$.

For an integrable vector or matrix function $f$, write
$\langle f\rangle_I=|I|^{-1}\int_I f$.
The matrices $X_I$ assigned to intervals in $\mathcal P$ need not equal
$\langle W_N\rangle_I$, since the orthonormal bases change during the
construction. The scalar averages satisfy the identities and estimates
in Lemma~\ref{lambda} below.

For $x\in I_0$, let $L(x)$ denote the unique interval in $E$
containing $x$. For each $i\in\{a,b\}$, define
\begin{equation*}
\Lambda_i(x)=
\begin{cases}
x_{i,k}^\#,
&L(x)\in E_{2k},\quad 1\le k<N,\\
x_{i,N-1},
&L(x)\in E_{2N}.
\end{cases}
\end{equation*}
Hence,
\begin{equation*}
\Lambda_i(x)^{-1/(p-1)}=
\begin{cases}
y_{i,k}^\#,
&L(x)\in E_{2k},\quad 1\le k<N,\\
{Q^{-1/(p-1)}}y_{i,N-1},
&L(x)\in E_{2N}.
\end{cases}
\end{equation*}

Then
\begin{equation}
W_N(x)=\Lambda_a(x){\bf a}_{L(x)}{\bf a}_{L(x)}^*
+\Lambda_b(x){\bf b}_{L(x)}{\bf b}_{L(x)}^*
\label{eq:truncation-pointwise-spectral-form}
\end{equation}
and
\begin{equation*}
W_N(x)^{-1/(p-1)}=\Lambda_a(x)^{-1/(p-1)}{\bf a}_{L(x)}{\bf a}_{L(x)}^*
+\Lambda_b(x)^{-1/(p-1)}{\bf b}_{L(x)}{\bf b}_{L(x)}^*.
\end{equation*}

The following scalar moment estimates will be used to bound the dyadic
$A_p$ characteristic of $W_N$.
\begin{lemma}\label{lambda}Let $0\le n<N-1$ and $I\in\mathcal P_{2n}\cup\mathcal P_{2n+1}$. We have
\begin{equation}
\frac1{|I|}\int_I\Lambda_i(x)\,dx
=x_{i,n},\qquad i\in\{a,b\}.
\label{eq:truncation-primal-moment}
\end{equation}
\begin{equation}\label{dual-id}
\frac1{|I|}\int_I(\Lambda_i(x))^{-1/(p-1)}\,dx
\leq
y_{i,n},\qquad i\in\{a,b\}.
\end{equation}

\end{lemma}
\begin{proof}
By Proposition~\ref{pro-1}(ii), the intervals in $E$ contained in
$I$ partition $I$. Those in $E_{2k}$ occupy a fraction
$2^{-(k-n)}$ for $n<k<N$, and those in $E_{2N}$ occupy a fraction
$2^{-(N-n-1)}$. These proportions hold for both
$I\in\mathcal P_{2n}$ and $I\in\mathcal P_{2n+1}$. Hence
\[
\frac1{|I|}\int_I\Lambda_i(x)\,dx
=\sum_{k=n+1}^{N-1}2^{-(k-n)}x_{i,k}^\#+2^{-(N-n-1)}x_{i,N-1}
=x_{i,n}.
\]
The last equality follows by  substituting $x_{i,k}^\#=2x_{i,k-1}-x_{i,k}$ from Proposition~\ref{W-pro}(iv). Similarly,
\begin{align*}
	\frac1{|I|}\int_I\Lambda_i(x)^{-1/(p-1)}\,dx
	&=\sum_{k=n+1}^{N-1}2^{-(k-n)}y_{i,k}^{\#}
	+2^{-(N-n-1)}Q^{-1/(p-1)}y_{i,N-1}\\
	&=y_{i,n}
	-2^{-(N-n-1)}
	\bigl(1-Q^{-1/(p-1)}\bigr)y_{i,N-1}
	\le y_{i,n}.
\end{align*}
Here we used $Q>1$ and the identity $y_{i,k}^\#=2y_{i,k-1}-y_{i,k}$.
\end{proof}

\subsection{The dyadic \texorpdfstring{$A_p$}{Ap} estimate}

The scalar estimates in Lemma~\ref{lambda} do not directly control the
matrix $A_p$ expression because the orthonormal bases change between
intervals. We first compare a fixed vector with the bases on the
stopping intervals, and then use the resulting estimates to bound the
two integrals in the $A_p$ expression.

For any matrix weight $W$ and any finite interval $I$ of positive
length, define
\[
\Phi_{p,I}(W)
:=\frac1{|I|}\int_I
\left(
\frac1{|I|}\int_I
\|W(x)^{1/p}W(y)^{-1/p}\|_{\mathrm{op}}^q\,dy
\right)^{p/q}dx.
\]
The dyadic matrix $A_p$ characteristic is
\[
[W]_{A_p}^{\mathrm{dy}}
:=\sup_{I\in\mathcal D}\Phi_{p,I}(W).
\]
In this section, we apply these definitions to $W_N$.

The next proposition bounds the dyadic matrix $A_p$ characteristic of
$W_N$ uniformly in the truncation depth $N$.

\begin{proposition}
\label{lem:dyadic-Ap-upper}
Let $p\ge 2$ and let $Q$ be sufficiently large. For every integer $N\ge2$, the finite
weight $W_N$ constructed in
Subsection~\ref{sec:finite-truncation-eigenvalues} satisfies
\begin{equation}
[W_N]_{A_p}^{\mathrm{dy}}\lesssim_p Q
\label{eq:dyadic-Ap-upper}
\end{equation}
with an implicit constant independent of $N$ and $Q$.
\end{proposition}
We prove the proposition after establishing two vector estimates and
their consequence for the auxiliary matrices used below.

\begin{lemma}
\label{lem:coordinate-leakage}
Let $0\le n<N-1$ and
$I\in\mathcal P_{2n}\cup\mathcal P_{2n+1}$. Let $J\in E$ satisfy
$J\subseteq I$. For every fixed vector
${\bf e}=u{\bf a}_I+v{\bf b}_I$, define $c_a(J)$ and $c_b(J)$ by
\begin{equation*}
{\bf e}=u{\bf a}_{I}+v{\bf b}_{I}=c_a(J){\bf a}_{J}+c_b(J){\bf b}_{J}.
\end{equation*}
Then, for each $s\in\{p,q\}$,
\[
|c_a(J)|^s\lesssim_p |u|^s+\delta_n^s|v|^s,
\qquad
|c_b(J)|^s\lesssim_p |v|^s+\delta_n^s|u|^s.
\]
\end{lemma}

\begin{proof}
Let $\alpha_J=\Psi_J-\Psi_I$ be the signed angle from
$({\bf a}_I,{\bf b}_I)$ to $({\bf a}_J,{\bf b}_J)$. By
Proposition~\ref{W-pro}(iii), $|\alpha_J|\lesssim_p\delta_n$ for
either parity of $I$. Since
\[
c_a(J)=u\cos\alpha_J+v\sin\alpha_J,
\qquad
c_b(J)=-u\sin\alpha_J+v\cos\alpha_J,
\]
the bounds $|\sin t|\le|t|$ and $|\cos t|\le1$, followed by
$(A+B)^s\le 2^{s-1}(A^s+B^s)$ for $A,B\ge0$, give
\begin{equation}
|c_a(J)|^s\lesssim_p |u|^s+\delta_n^s|v|^s,
\qquad
|c_b(J)|^s\lesssim_p |v|^s+\delta_n^s|u|^s.
\label{eq:section-three-coordinate-leakage}
\end{equation}
This proves the desired estimates.
\end{proof}

The choice of the rotation angles gives the identities
\[
\delta_n^p x_{b,n}=x_{a,n},
\qquad
\delta_n^q y_{a,n}=y_{b,n}.
\]
They allow us to control the terms introduced by the rotations.
In the first estimate below, the factor $\delta_n^p$ compensates
for the larger scalar value $x_{b,n}$.
In the second estimate, the factor $\delta_n^q$ plays the same
role for $y_{a,n}$.
Together with the scalar moment bounds, these identities give
the required estimates in the fixed basis
$({\bf a}_I,{\bf b}_I)$.

\begin{lemma}
\label{lem:fixed-frame-vector-estimates}
Let $p\ge 2$ and let $Q$ be sufficiently large. For every $0\le n<N-1$, every $I\in\mathcal P_{2n}\cup\mathcal P_{2n+1}$, and every vector ${\bf e}=u{\bf a}_I+v{\bf b}_I$, the following estimates hold.
\begin{align}
\frac1{|I|}\int_I
|W_N(x)^{1/p} {\bf e}|^p\,dx
&\lesssim_p x_{a,n}|u|^p+x_{b,n}|v|^p,
\label{eq:section-three-primal-fixed-frame}\\
\frac1{|I|}\int_I
|W_N(x)^{-1/p} {\bf e}|^q\,dx
&\lesssim_p y_{a,n}|u|^q+y_{b,n}|v|^q.
\label{eq:section-three-dual-fixed-frame}
\end{align}
\end{lemma}
\begin{proof}
For $x\in I$, apply Lemma~\ref{lem:coordinate-leakage} with $J=L(x)$. The spectral representation \eqref{eq:truncation-pointwise-spectral-form} gives
\begin{align*}
W_N(x)^{1/p}{\bf e}
&=c_a(L(x))\Lambda_a(x)^{1/p}{\bf a}_{L(x)}
+c_b(L(x))\Lambda_b(x)^{1/p}{\bf b}_{L(x)}.
\end{align*}
Orthogonality and convexity give
\begin{equation*}
|W_N(x)^{1/p} {\bf e}|^p
\lesssim_p
\Lambda_a(x)|c_a(L(x))|^p+\Lambda_b(x)|c_b(L(x))|^p.
\end{equation*}
By Proposition~\ref{W-pro}, $\delta_n\le1$,
$\delta_n^p x_{b,n}=x_{a,n}$, and
$\delta_n^q y_{a,n}=y_{b,n}$. These identities also give
$\delta_n^p x_{a,n}\le x_{b,n}$ and
$\delta_n^q y_{b,n}\le y_{a,n}$.
Thus \eqref{eq:section-three-coordinate-leakage} and
\eqref{eq:truncation-primal-moment} yield
\[
\begin{aligned}
\frac1{|I|}\int_I|W_N(x)^{1/p} {\bf e}|^p\,dx
&\lesssim_p
x_{a,n}\bigl(|u|^p+\delta_n^p|v|^p\bigr)
+x_{b,n}\bigl(|v|^p+\delta_n^p|u|^p\bigr)\\
&\lesssim_p x_{a,n}|u|^p+x_{b,n}|v|^p.
\end{aligned}
\]

Since $q\le2$, the inequality
$(A+B)^{q/2}\le A^{q/2}+B^{q/2}$ for $A,B\ge0$ gives
\begin{equation*}
|W_N(x)^{-1/p} {\bf e}|^q
\le
(\Lambda_a(x))^{-1/(p-1)}|c_a(L(x))|^q
+(\Lambda_b(x))^{-1/(p-1)}|c_b(L(x))|^q.
\end{equation*}
Combining this with \eqref{dual-id} and \eqref{eq:section-three-coordinate-leakage}, we obtain
\[
\begin{aligned}
\frac1{|I|}\int_I|W_N(x)^{-1/p} {\bf e}|^q\,dx
&\lesssim_p
y_{a,n}\bigl(|u|^q+\delta_n^q|v|^q\bigr)
+y_{b,n}\bigl(|v|^q+\delta_n^q|u|^q\bigr)\\
&\lesssim_p y_{a,n}|u|^q+y_{b,n}|v|^q.
\end{aligned}
\]
This proves both estimates.
\end{proof}

For $0\le n<N-1$, fix $I\in\mathcal P_{2n}\cup\mathcal P_{2n+1}$. Set
\begin{equation}
V_I=x_{a,n}^{1/p}{\bf a}_I{\bf a}_I^*+x_{b,n}^{1/p}{\bf b}_I{\bf b}_I^*,
\qquad
V_I^\sharp=y_{a,n}^{1/q}{\bf a}_I{\bf a}_I^*+y_{b,n}^{1/q}{\bf b}_I{\bf b}_I^*.
\label{eq:section-three-reducing-matrices}
\end{equation}
We use the following upper estimates for these auxiliary matrices.
\begin{corollary}
\label{cor:dyadic-reducing-estimates}
Let $p\ge 2$ and let $Q$ be sufficiently large. For every $0\le n<N-1$ and $I\in\mathcal P_{2n}\cup\mathcal P_{2n+1}$,
\begin{equation}
\frac1{|I|}\int_I
\|W_N(x)^{1/p}V_I^{-1}\|_{\mathrm{op}}^p\,dx
\lesssim_p1,
\qquad
\frac1{|I|}\int_I
\|(V_I^\sharp)^{-1}W_N(y)^{-1/p}\|_{\mathrm{op}}^q\,dy
\lesssim_p1.
\label{eq:section-three-reducing-goal}
\end{equation}
\end{corollary}

\begin{proof}
Applying
\eqref{eq:section-three-primal-fixed-frame} to
$x_{a,n}^{-1/p} {\bf a}_I$ and $x_{b,n}^{-1/p} {\bf b}_I$ gives
\[
\frac1{|I|}\int_I
\left(
|W_N(x)^{1/p}V_I^{-1} {\bf a}_I|^p
+|W_N(x)^{1/p}V_I^{-1} {\bf b}_I|^p
\right)dx
\lesssim_p1.
\]
The bound
$\|S\|_{\mathrm{op}}^p\lesssim_p|S{\bf a}_I|^p+|S{\bf b}_I|^p$
proves the first estimate in \eqref{eq:section-three-reducing-goal}.

Since $W_N(y)^{-1/p}$ and $(V_I^\sharp)^{-1}$ are symmetric, taking adjoints gives
\[
\|(V_I^\sharp)^{-1}W_N(y)^{-1/p}\|_{\mathrm{op}}
=\|W_N(y)^{-1/p}(V_I^\sharp)^{-1}\|_{\mathrm{op}}.
\]
Apply \eqref{eq:section-three-dual-fixed-frame} to $y_{a,n}^{-1/q}{\bf a}_I$ and $y_{b,n}^{-1/q}{\bf b}_I$. The same bound by the two column norms gives
\begin{align*}
&\frac1{|I|}\int_I\|(V_I^\sharp)^{-1}W_N(y)^{-1/p}\|_{\mathrm{op}}^q\,dy\\
&\lesssim_p\frac1{|I|}\int_I
\Bigl(|W_N(y)^{-1/p}(V_I^\sharp)^{-1}{\bf a}_I|^q
+|W_N(y)^{-1/p}(V_I^\sharp)^{-1}{\bf b}_I|^q\Bigr)\,dy\lesssim_p1,
\end{align*}
which completes the estimate.
\end{proof}

\begin{proof}[Proof of Proposition~\ref{lem:dyadic-Ap-upper}]
Fix $I\in\mathcal D$. If $I\subseteq J$ for some $J\in E$, then
$W_N$ is constant on $I$ and
\begin{equation*}
\Phi_{p,I}(W_N)=1.
\end{equation*}

Otherwise, Proposition~\ref{pro-1}(iii) gives
$I\in\mathcal P_{2n}\cup\mathcal P_{2n+1}$ for some $0\le n\le N-2$.
The auxiliary matrices in
\eqref{eq:section-three-reducing-matrices} are diagonal in
$({\bf a}_I,{\bf b}_I)$. For every $x,y\in I$, we factor
\begin{equation*}
\begin{aligned}
W_N(x)^{1/p}W_N(y)^{-1/p}
={}&\bigl(W_N(x)^{1/p}V_I^{-1}\bigr)
(V_IV_I^\sharp)\times
\bigl((V_I^\sharp)^{-1}W_N(y)^{-1/p}\bigr).
\end{aligned}
\end{equation*}
Taking operator norms and separating the $x$- and $y$-dependent factors gives
\begin{equation}
\begin{aligned}
\Phi_{p,I}(W_N)
\le{}&\|V_IV_I^\sharp\|_{\mathrm{op}}^p
\left(
\frac1{|I|}\int_I
\|W_N(x)^{1/p}V_I^{-1}\|_{\mathrm{op}}^p\,dx
\right)\\
&\times
\left(
\frac1{|I|}\int_I
\|(V_I^\sharp)^{-1}W_N(y)^{-1/p}\|_{\mathrm{op}}^q\,dy
\right)^{p/q}.
\end{aligned}
\label{eq:section-three-roudenko-factorization}
\end{equation}
The two integral factors are bounded by
Corollary~\ref{cor:dyadic-reducing-estimates}.  Since $V_I$ and
$V_I^\sharp$ are diagonal in the same ordered orthonormal basis, and since
$p/q=p-1$, their middle factor is
\begin{equation*}
\begin{aligned}
\|V_IV_I^\sharp\|_{\mathrm{op}}^p
&=\max\{
x_{a,n}y_{a,n}^{p/q},
x_{b,n}y_{b,n}^{p/q}
\}\\
&=\max\{
x_{a,n}y_{a,n}^{p-1},
x_{b,n}y_{b,n}^{p-1}
\}
=Q.
\end{aligned}
\end{equation*}
The last equality follows from Proposition~\ref{W-pro}(iv). It follows from
\eqref{eq:section-three-roudenko-factorization} and Corollary~\ref{cor:dyadic-reducing-estimates} that
$
\Phi_{p,I}(W_N)\lesssim_p Q.
$
Taking the supremum over $I\in\mathcal D$ proves \eqref{eq:dyadic-Ap-upper}.
\end{proof}

\section{Construction of the test function}\label{sec:test-function}
We construct a finite vector-valued test function ${\bf f}_N$ whose
averages on the continuing intervals are parallel to the vectors
${\bf a}_I$. This choice keeps its weighted norm small, while the
rotations produce Haar coefficients that will be used in the lower
bound in Section~\ref{sec:haar-lower}.
We retain the parameters from Section~\ref{sec:2}, with
$2\le p<\infty$ and $Q$ sufficiently large.

We choose the amplitude $\ell^n$ to compensate for the decreasing
rotation angles, since
\[
\ell^n\tan\theta_n=\delta_0.
\]
As Proposition~\ref{prop:test-haar-coefficients} shows, the Haar
coefficients on the continuing intervals in the even generations
therefore have the same size after division by the square root
of the interval length.

Recall that for every dyadic interval $J\subseteq I_0=[0,1)$,
${\bf a}_J$ is the first vector in the ordered orthonormal basis
constructed in Section~\ref{sec:2}. Set
\begin{equation*}
{\bf f}_{I_0}={\bf a}_{I_0}.
\end{equation*}
Fix $0\le n\le N-2$ and $I\in\mathcal P_{2n}$. Assume that ${\bf f}_I=\ell^n{\bf a}_I$.
For its children $I_\pm\in\mathcal P_{2n+1}$, use the ordered
orthonormal bases from Subsection~\ref{rotation}. To preserve the
averaging identity $2{\bf f}_I={\bf f}_{I_+}+{\bf f}_{I_-}$, define
\begin{equation}
{\bf f}_{I_+}=\frac{\ell^n}{\cos\theta_n}\,{\bf a}_{I_+},\qquad
{\bf f}_{I_-}=\frac{\ell^n}{\cos\theta_n}\,{\bf a}_{I_-}.
\label{eq:test-generation-rotation-values}
\end{equation}
Since ${\bf a}_{I_+}+{\bf a}_{I_-}=2\cos\theta_n\,{\bf a}_I$, these values have the required average.
For the stretching step, set
\begin{equation}
t_n=\frac{2}{\cos\theta_n}-\ell,\qquad 0\le n\le N-2,
\label{eq:test-final-value}
\end{equation}
and assign
\begin{equation}
\begin{aligned}
{{\bf f}_{I_{++}}}=\ell^{n+1}{\bf a}_{I_{++}},& \qquad {{\bf f}_{I_{+-}}}=\ell^{n}t_n\,{\bf a}_{I_{+-}},\\
{{\bf f}_{I_{-+}}}=\ell^{n}t_n\,{\bf a}_{I_{-+}},&\qquad  {{\bf f}_{I_{--}}}=\ell^{n+1}{\bf a}_{I_{--}}.
\end{aligned}
\label{eq:test-generation-stretch-values}
\end{equation}
The grandchildren use the same ordered orthonormal bases as their
parents. Thus, for either sign,
\[
\frac{{\bf f}_{I_{\pm+}}+{\bf f}_{I_{\pm-}}}{2}
=\frac{\ell^n(\ell+t_n)}2\,{\bf a}_{I_\pm}
=\frac{\ell^n}{\cos\theta_n}\,{\bf a}_{I_\pm}
={\bf f}_{I_\pm}.
\]

For $n<N-2$, the outer grandchildren belong to
$\mathcal P_{2n+2}$. At $n=N-2$, they belong to $E_{2N}$ and have
value $\ell^{N-1}{\bf a}_J$. The inner
grandchildren belong to $E_{2n+2}$ at every step. When $N=2$, all
four grandchildren belong to $E$ after the first step.
The scalar amplitudes $f_J$ in ${\bf f}_J=f_J{\bf a}_J$ are therefore
\[
f_J=\begin{cases}
\ell^{n-1}t_{n-1},&J\in E_{2n},\quad 1\le n\le N-1,\\
\ell^{N-1},&J\in E_{2N},\\
\ell^n,&J\in\mathcal P_{2n},\quad 0\le n\le N-2,\\
\ell^n/\cos\theta_n,&J\in\mathcal P_{2n+1},\quad 0\le n\le N-2.
\end{cases}
\]

Using the interval $L(x)\in E$ defined in Section~\ref{sec:2}, set,
for $x\in I_0$,
\begin{equation}
\begin{aligned}
{\bf f}_N(x)&=\sum_{J\in E}f_J{\bf a}_J\one_J(x)=f_{L(x)}{\bf a}_{L(x)}\\
&={\sum_{n=1}^{N-1}}
\sum_{J\in E_{2n}}
\ell^{n-1}t_{n-1}
{\bf a}_{J}\one_{J}(x)
+
\sum_{J\in E_{2N}}
\ell^{N-1}{\bf a}_{J}\one_J(x).
\end{aligned}
\label{eq:explicit-test-function}
\end{equation}
The averaging identities at the rotation and stretching steps imply,
by induction from the intervals in $E$ towards the root, that
\[
\langle{\bf f}_N\rangle_J=\frac1{|J|}\int_J{\bf f}_N(x)\,dx={\bf f}_J=f_J{\bf a}_J
\]
for every interval $J$ in $\mathcal P\cup E$ and every odd-generation
interval used in the construction. Thus the assigned vectors ${\bf f}_J$
are the actual conditional averages of ${\bf f}_N$. In particular,
$\langle{\bf f}_N\rangle_{I_0}={\bf a}_{I_0}$, so ${\bf f}_N\ne0$.
Let $\mathcal F_m$ be the sigma-algebra generated by $\mathcal D_m$.
Then
\[
\bigl(\mathbb E({\bf f}_N\mid\mathcal F_m)\bigr)_{m=0}^{2N-2}
\]
is a finite martingale with terminal value ${\bf f}_N$.

The next lemma shows that the factors $t_n$ on the stopping intervals
are comparable to $2-\ell$, uniformly in $n$.
\begin{lemma}
\label{lem:test-final-amplitude}Assume $0<\delta_0\le\frac12(2-\ell)$.
Uniformly for \(0\le n\le N-2\),
\begin{equation}
t_n\asympp 2-\ell\asympp Q^{-1/(p-1)}.
\label{eq:t-asymptotics}
\end{equation}
\end{lemma}

\begin{proof}
Since $\theta_n=\arctan\delta_n$ by \eqref{def-delta},
$(\cos\theta_n)^{-1}=\sqrt{1+\delta_n^2}$, and hence
$$
t_n
=2\sqrt{1+\delta_n^2}-\ell
=(2-\ell)
+\frac{2\delta_n^2}{\sqrt{1+\delta_n^2}+1}.
$$
Hence
\begin{equation*}
2-\ell
\le t_n
\le (2-\ell)+\delta_n^2.
\end{equation*}
Since $\ell\ge3/2$ and $\delta_n\le\delta_0$, the hypothesis gives
\[
\delta_n^2\le\delta_0^2\le\frac{(2-\ell)^2}{4}\le\frac{2-\ell}{8}.
\]
Thus $2-\ell\le t_n\le\frac98(2-\ell)$. The second equivalence in \eqref{eq:t-asymptotics} follows from \eqref{eq:d-tau-asymptotics2}.
\end{proof}

The following estimate separates the contribution of the stopping
intervals from that of the terminal intervals in $E_{2N}$.
\begin{proposition}
Under the assumptions of Lemma~\ref{lem:test-final-amplitude}, for every $N\ge2$,
\begin{equation}
\|{\bf f}_N\|_{L^p(W_N)}^p
\lesssim_p
{x_{a,0}}Q^{-1/(p-1)}
+{x_{a,0}}\left(\frac{r_b}{2}\right)^{N-1}.
\label{eq:input-norm-bound}
\end{equation}

\end{proposition}

\begin{proof}
By \eqref{eq:truncation-pointwise-spectral-form} and \eqref{eq:explicit-test-function},
\[
W_N(x)^{1/p}{\bf f}_N(x)=\Lambda_a(x)^{1/p}f_{L(x)}{\bf a}_{L(x)}.
\]
Proposition~\ref{pro-1}(ii), applied to $I_0$, gives
\[
\sum_{I\in E_{2n}}|I|=2^{-n}\quad(1\le n<N),\qquad
\sum_{I\in E_{2N}}|I|=2^{1-N}.
\]
Consequently,
\begin{align*}
\|{\bf f}_N\|_{L^p(W_N)}^p
&=\int_{I_0}\Lambda_a(x)|f_{L(x)}|^p\,dx\\
&=\sum_{n=1}^{N-1}2^{-n}x_{a,n}^\#\ell^{p(n-1)}t_{n-1}^p
+2^{1-N}x_{a,N-1}\ell^{p(N-1)}.
\end{align*}
Substituting the formulas from Proposition~\ref{W-pro}(i) and using $\ell^p=r_a^{p-1}r_b$, we obtain the exact identity
\begin{equation}\label{eq:input-norm-identity}
\|{\bf f}_N\|_{L^p(W_N)}^p
=x_{a,0}\Bigg[\frac{2-r_a^{-(p-1)}}2
\sum_{j=0}^{N-2}\left(\frac{r_b}{2}\right)^j t_j^p
+\left(\frac{r_b}{2}\right)^{N-1}\Bigg].
\end{equation}
By Lemma~\ref{lem:test-final-amplitude} and \eqref{eq:d-tau-asymptotics},
\[
t_j^p\lesssim_p Q^{-p/(p-1)},\qquad
\left(1-\frac{r_b}{2}\right)^{-1}\asymp_p Q.
\]
Since $1<2-r_a^{-(p-1)}<2$, the first term in \eqref{eq:input-norm-identity} is bounded by
\[
C_p x_{a,0}Q^{-p/(p-1)}
\sum_{j=0}^{N-2}\left(\frac{r_b}{2}\right)^j
\lesssim_p x_{a,0}Q^{-1/(p-1)}.
\]
The second term is exactly the contribution of $E_{2N}$. Keeping this
term gives \eqref{eq:input-norm-bound}.
\end{proof}

\section{A lower bound for \texorpdfstring{$\mathbb S-\mathbb S^*$}{S minus its adjoint}}\label{sec:haar-lower}

\subsection{Haar shift and Haar coefficients}

Let $\mathcal D_{\rm odd}=\bigcup_{n\ge0}\mathcal D_{2n+1}$.
With $I_+$ denoting the left child, set
\begin{equation}\label{eq:haar-normalization}
h_I=|I|^{-1/2}(\one_{I_+}-\one_{I_-}).
\end{equation}
Let ${\bf f},{\bf g}\colon I_0\to\mathbb R^2$ be functions with finite
Haar expansions, including possible constant terms. For vectors
$u,v\in\mathbb R^2$, $\langle u,v\rangle$ denotes their Euclidean
inner product. For a scalar Haar function $h_I$, write
$\langle{\bf f},h_I\rangle=\int_I{\bf f}(x)h_I(x)\,dx$.
Following the odd-generation restriction used in
\cite[Section 6]{domelevo-petermichl-treil-volberg-2024}, define
the Haar shift $\mathbb S$ and its unweighted adjoint by
\begin{align*}
\mathbb S{\bf f}&=\sum_{I\in\mathcal D_{\rm odd}}
\langle{\bf f},h_I\rangle(h_{I_+}-h_{I_-}),\\
\mathbb S^*{\bf f}&=\sum_{I\in\mathcal D_{\rm odd}}
\langle{\bf f},h_{I_+}-h_{I_-}\rangle h_I.
\end{align*}
Orthogonality of the Haar functions gives, by rearranging finite sums,
\[
\int_{I_0}\langle T{\bf f}(x),{\bf g}(x)\rangle\,dx=-\int_{I_0}\langle{\bf f}(x),T{\bf g}(x)\rangle\,dx,\qquad
T=\mathbb S-\mathbb S^*.
\]
We use this identity below to exchange the two test functions.

We first compute the Haar coefficients of ${\bf f}_N$ on the intervals
used in the construction.
\begin{proposition}\label{prop:test-haar-coefficients}
Each \(J\in\mathcal D\) falls into one of the following three cases.
\begin{enumerate}
\item[{\rm(i)}] If \(J\in\mathcal P_{2n}\) for some {\(0\le n\le N-2\)}, then
\begin{equation}
\langle {\bf f}_N,h_J\rangle
=
|J|^{1/2}\delta_0 {\bf b}_J.
\label{eq:rotation-haar-coefficient}
\end{equation}

\item[{\rm(ii)}]  If \(J\in\mathcal P_{2n+1}\) for some {\(0\le n\le N-2\)}, let
\(K\in\mathcal P_{2n}\) be its unique parent.  Then \(J=K_+\) or
\(J=K_-\), and
\begin{equation*}
\begin{aligned}
\langle {\bf f}_N,h_{K_+}\rangle
&=
\frac{|K|^{1/2}}{\sqrt2}\ell^n
\frac{\ell\cos\theta_n-1}{\cos\theta_n}\,{\bf a}_{K_+},\\
\langle {\bf f}_N,h_{K_-}\rangle
&=
-\frac{|K|^{1/2}}{\sqrt2}\ell^n
\frac{\ell\cos\theta_n-1}{\cos\theta_n}\,{\bf a}_{K_-}.
\end{aligned}
\end{equation*}

\item[{\rm(iii)}]  If
$
J\notin
\bigcup_{n=0}^{N-2}
\bigl(\mathcal P_{2n}\cup\mathcal P_{2n+1}\bigr),
$
then
\begin{equation}
\langle {\bf f}_N,h_J\rangle=0.
\label{eq:all-other-haar-coefficients-zero}
\end{equation}
{In particular, this case includes $E_{2N}$ and every $J$ contained
in an interval $I\in E$, on which ${\bf f}_N$ is constant.}
\end{enumerate}
\end{proposition}

\begin{proof}
\textit{Case 1. \(J\in\mathcal P_{2n}\), {\(0\le n\le N-2\)}.}
The Haar normalization \eqref{eq:haar-normalization} and the averages
in \eqref{eq:test-generation-rotation-values} give
\[
\begin{aligned}
\langle {\bf f}_N,h_J\rangle
&=
\frac{|J|^{1/2}}2
\left(
\langle {\bf f}_N\rangle_{J_+}
-\langle {\bf f}_N\rangle_{J_-}
\right)\\
&=
\frac{|J|^{1/2}\ell^n}{2\cos\theta_n}
\left({\bf a}_{J_+}-{\bf a}_{J_-}\right)\\
&=
|J|^{1/2}\ell^n
\frac{\sin\theta_n}{\cos\theta_n}\,{\bf b}_J.
\end{aligned}
\]
Here we used ${\bf a}_{J_+}-{\bf a}_{J_-}=2\sin\theta_n{\bf b}_J$.
Since $\tan\theta_n=\delta_n=\delta_0\ell^{-n}$ by
Proposition~\ref{W-pro}(ii), this proves
\eqref{eq:rotation-haar-coefficient}.

\medskip
\noindent
\textit{Case 2. \(J\in\mathcal P_{2n+1}\), {\(0\le n\le N-2\)}.}
Let $K\in\mathcal P_{2n}$ be the parent of $J$. By
\eqref{eq:test-final-value}, we have
$(\ell-t_n)/2=(\ell\cos\theta_n-1)/\cos\theta_n$.
The averages in \eqref{eq:test-generation-stretch-values} therefore give
\[
\begin{aligned}
\langle {\bf f}_N,h_{K_+}\rangle
&=
\frac{|K|^{1/2}}{\sqrt2}\ell^n
\frac{\ell\cos\theta_n-1}{\cos\theta_n}\,{\bf a}_{K_+},\\
\langle {\bf f}_N,h_{K_-}\rangle
&=
-\frac{|K|^{1/2}}{\sqrt2}\ell^n
\frac{\ell\cos\theta_n-1}{\cos\theta_n}\,{\bf a}_{K_-}.
\end{aligned}
\]
The sign changes because the values with scalar factors
$\ell^{n+1}$ and $\ell^n t_n$ occur in the opposite order on the
children of $K_-$.

\medskip
\noindent
\textit{Case 3. All remaining \(J\in\mathcal D\).}
{If $J\notin\bigcup_{n=0}^{N-2}
(\mathcal P_{2n}\cup\mathcal P_{2n+1})$, then either $J\in E$ or
$J\subset I$ for some $I\in E$. This includes $E_{2N}$. In both
situations ${\bf f}_N$ is constant on $J$ by
\eqref{eq:explicit-test-function}. Therefore
$\langle {\bf f}_N,h_J\rangle=0$ and
\eqref{eq:all-other-haar-coefficients-zero} holds.}
\end{proof}

\subsection{Lower bound for the Haar shift}
The following theorem is the main result of this subsection, which provides
a lower bound for the difference of Haar shifts.

\begin{theorem}\label{thm:main}
Let \(1<p<\infty\). There exist constants \(c_p,C_p>0\) such
that for all sufficiently large \(Q\) there is a \(2\times2\) matrix
weight \(W\) on \(I_0\) with \([W]_{A_p}^{\mathrm{dy}}\le C_pQ\) and a function
\(0\ne f\in L^p(W)\) satisfying
\begin{equation*}
\|(\mathbb S-\mathbb S^*) f\|_{L^p(W)}
\ge c_pQ^{1+1/[p(p-1)]}\|f\|_{L^p(W)}.
\end{equation*}
\end{theorem}

Recall that $L(x)\in E$ denotes the unique
interval in $E$ containing $x$. 
We next estimate $(\mathbb S-\mathbb S^*){\bf f}_N$ in the direction
of ${\bf b}_L$.
\begin{lemma}
\label{lem:pathwise-antisymmetric}
Let $p\ge2$ and let $Q$ be sufficiently large. Suppose that
$0<\delta_0\le(2-\ell)/2$ and $N\ge6$. For every $4\le k\le N-2$
and $L\in E_{2k}$, we have, for almost every $x\in L$,
\begin{equation}
\bigl\langle(\mathbb S-\mathbb S^*){\bf f}_N(x),{\bf b}_{L(x)}\bigr\rangle
\le -c_p k\delta_0.
\label{eq:pathwise-output2}
\end{equation}
\end{lemma}
\begin{proof}
Choose $Q$ sufficiently large that $\ell\ge19/10$. Then
$\delta_0\le(2-\ell)/2\le1/20$, and
\[
\sec\theta_n=\sqrt{1+\delta_n^2}
\le\sqrt{1+\delta_0^2}<\ell.
\]
In particular, $0<\ell-\sec\theta_n\le\ell-1$.
Fix $L\in E_{2k}$ and $x\in L$. For $0\le n\le k-1$, let
$J_{2n}\in\mathcal P_{2n}$ and $J_{2n+1}\in\mathcal P_{2n+1}$
be the ancestors containing $L$, and set $J_{2k}=L$. Thus $J_{2k-1}$
is the parent of $L$, and
\[
x\in J_{2k}\subset J_{2k-1}\subset\cdots\subset J_0=I_0.
\]
Write $\sigma_n=1$ when $J_{2n+1}=(J_{2n})_+$ and
$\sigma_n=-1$ otherwise. The stretch steps preserve the ordered
orthonormal basis, so
\[
\psi_n=\Psi_L-\Psi_{J_{2n}}
=\sum_{m=n}^{k-1}\sigma_m\theta_m
\quad(0\le n\le k-1),\qquad \psi_k=0.
\]
Since $\theta_m\le\delta_0\ell^{-m}$ and $\ell\ge3/2$,
\[
|\psi_n|\le\frac{\ell}{\ell-1}\delta_0\ell^{-n},
\qquad \sum_{n=1}^{k-1}|\psi_n|^2\le C\delta_0^2.
\]

By Proposition~\ref{prop:test-haar-coefficients}, only the odd
ancestors $J_{2n+1}$, $0\le n\le k-1$, can contribute at $x$.
Hence
\begin{align*}
\mathbb S{\bf f}_N(x)
&=\sum_{n=0}^{k-1}\langle{\bf f}_N,h_{J_{2n+1}}\rangle
(h_{(J_{2n+1})_+}-h_{(J_{2n+1})_-})(x),\\
\mathbb S^*{\bf f}_N(x)
&=\sum_{n=0}^{k-1}
\langle{\bf f}_N,h_{(J_{2n+1})_+}-h_{(J_{2n+1})_-}\rangle
h_{J_{2n+1}}(x).
\end{align*}
For $0\le n\le k-2$, the interval $J_{2n+2}$ is an outer
grandchild of $J_{2n}$. The last interval $L$ is an inner grandchild
of $J_{2k-2}$. Therefore
\[
h_{J_{2n+1}}(x)=
\begin{cases}
\sigma_n|J_{2n+1}|^{-1/2},&0\le n\le k-2,\\
-\sigma_{k-1}|J_{2k-1}|^{-1/2},&n=k-1,
\end{cases}
\]
and
\[
(h_{(J_{2n+1})_+}-h_{(J_{2n+1})_-})(x)=
\begin{cases}
\sqrt2\sigma_n\sigma_{n+1}|J_{2n+1}|^{-1/2},&0\le n\le k-2,\\
-\sigma_{k-1}h_L(x),&n=k-1.
\end{cases}
\]
For $0\le n\le k-1$, Proposition~\ref{prop:test-haar-coefficients}
gives
\[
\langle{\bf f}_N,h_{J_{2n+1}}\rangle
=
\sigma_n|J_{2n+1}|^{1/2}
\ell^n(\ell-\sec\theta_n){\bf a}_{J_{2n+1}}.
\]
For $0\le n\le k-2$, multiplication by the corresponding
Haar value therefore gives
\[
\begin{aligned}
	&\langle{\bf f}_N,h_{J_{2n+1}}\rangle
	(h_{(J_{2n+1})_+}-h_{(J_{2n+1})_-})(x)\\
	&\qquad=
	\sqrt2\,\ell^n(\ell-\sec\theta_n)
	\sigma_{n+1}{\bf a}_{J_{2n+1}}.
\end{aligned}
\]
Here the two factors $\sigma_n$ cancel.
At the last ancestor, the same multiplication gives
\[
-|J_{2k-1}|^{1/2}\ell^{k-1}
(\ell-\sec\theta_{k-1})h_L(x){\bf a}_L.
\]

For the adjoint, the stopping child has zero Haar coefficient.
At each of the first $k-1$ ancestors, the continuing child
therefore contributes
\[
\frac{\delta_0}{\sqrt2}{\bf b}_{J_{2n+2}},
\qquad 0\le n\le k-2.
\]
At the last ancestor, $x$ lies in the stopping child.
The Haar value changes sign, and the continuing sibling
has the same basis as $L$.
Its contribution is consequently
$
-\frac{\delta_0}{\sqrt2}{\bf b}_L.
$

Substituting the Haar coefficients gives
\begin{align}
\mathbb S{\bf f}_N(x)
&=\sqrt2\sum_{n=0}^{k-2}\ell^n(\ell-\sec\theta_n)
\sigma_{n+1}{\bf a}_{J_{2n+1}}-{\bf r}(x),
\label{eq:pathwise-B}\\
\mathbb S^*{\bf f}_N(x)
&=\frac{\delta_0}{\sqrt2}
\left(\sum_{n=1}^{k-1}{\bf b}_{J_{2n}}-{\bf b}_L\right),
\label{eq:pathwise-D}
\end{align}
where
\[
{\bf r}(x)=|J_{2k-1}|^{1/2}\ell^{k-1}
(\ell-\sec\theta_{k-1})h_L(x){\bf a}_L.
\]
The assumption $k\le N-2$ ensures that
the sibling of $L$ still belongs to $\mathcal P$, so its Haar
coefficient is included in \eqref{eq:pathwise-D}.

The definitions of the rotated bases give
\[
\langle{\bf b}_{J_{2n}},{\bf b}_L\rangle=\cos\psi_n,
\qquad
\langle{\bf a}_{J_{2n+1}},{\bf b}_L\rangle=-\sin\psi_{n+1}.
\]
We now estimate the two projections separately.
The projection of $\mathbb S^*{\bf f}_N$ onto ${\bf b}_L$
provides the main positive term, which enters
$\mathbb S-\mathbb S^*$ with a negative sign.
For $\mathbb S{\bf f}_N$, only an upper bound on its projection
is needed.
Any negative contribution strengthens the desired estimate.
Using $\cos t\ge1-t^2/2$, we obtain
\begin{align*}
\langle\mathbb S^*{\bf f}_N(x),{\bf b}_L\rangle
&=\frac{\delta_0}{\sqrt2}
\left(\sum_{n=1}^{k-1}\cos\psi_n-1\right)\\
&\ge\frac{\delta_0}{\sqrt2}(k-2)
-\frac{\delta_0}{2\sqrt2}\sum_{n=1}^{k-1}|\psi_n|^2\\
&\ge\frac{\delta_0}{\sqrt2}(k-2)-C\delta_0^3.
\end{align*}

For the other term, the inequality $\arctan t\ge t-t^3/3$ for
$t\ge0$ gives
\begin{align*}
\sigma_m\psi_m
&\ge\theta_m-\sum_{j=m+1}^{k-1}\theta_j\\
&\ge-\frac{2-\ell}{\ell-1}\delta_0\ell^{-m}
-\frac{\delta_0^3}{3}\ell^{-3m}.
\end{align*}
Also $|\psi_m|\le3\delta_0<\pi/2$. If $\sigma_m\psi_m\le0$,
then $\sin(\sigma_m\psi_m)\ge\sigma_m\psi_m$. Otherwise its sine
is nonnegative. Thus
\[
\sigma_m\sin\psi_m
\ge-\frac{2-\ell}{\ell-1}\delta_0\ell^{-m}
-\frac{\delta_0^3}{3}\ell^{-3m}.
\]
Since ${\bf r}(x)\perp{\bf b}_L$, \eqref{eq:pathwise-B} yields
\begin{align*}
\langle\mathbb S{\bf f}_N(x),{\bf b}_L\rangle
&=-\sqrt2\sum_{m=1}^{k-1}\ell^{m-1}
(\ell-\sec\theta_{m-1})\sigma_m\sin\psi_m\\
&\le\frac{\sqrt2}{\ell}(k-1)(2-\ell)\delta_0
+\frac{\sqrt2(\ell-1)}{3}\delta_0^3
\sum_{m=1}^{k-1}\ell^{-2m-1}\\
&\le\frac{\sqrt2}{\ell}(k-1)(2-\ell)\delta_0+C\delta_0^3.
\end{align*}
Combining the two estimates gives
\begin{equation}
\begin{aligned}
\bigl\langle(\mathbb S-\mathbb S^*){\bf f}_N(x),{\bf b}_L\bigr\rangle
\le{}&-\frac{\delta_0}{\sqrt2}(k-2)\\
&+\frac{\sqrt2}{\ell}(k-1)(2-\ell)\delta_0+C\delta_0^3.
\end{aligned}
\label{eq:pathwise-total}
\end{equation}
Here the combined cubic error can be bounded by $3\delta_0^3$.
Indeed, for $\ell\ge3/2$,
\[
\sum_{n=1}^{k-1}|\psi_n|^2\le\frac{36}{5}\delta_0^2,
\qquad
\frac{\sqrt2(\ell-1)}3\sum_{m=1}^{\infty}\ell^{-2m-1}
=\frac{\sqrt2}{3\ell(\ell+1)}\le\frac{4\sqrt2}{45},
\]
and $18/(5\sqrt2)+4\sqrt2/45<3$.
Set $A=(k-2)\delta_0/\sqrt2$. For $k\ge4$, the parameter bounds
fixed above give the following ratios of the two positive errors to $A$
\[
\frac{3(2-\ell)}{\ell}\le\frac3{19}<\frac14,
\qquad
\frac{3\sqrt2\delta_0^2}{k-2}\le\frac{3\sqrt2}{800}<\frac14.
\]
Each positive error is therefore at most $A/4$, so the left-hand
side of \eqref{eq:pathwise-total} is at most
$-A/2\le-k\delta_0/(4\sqrt2)$, uniformly in $k$ and $N$.
This proves \eqref{eq:pathwise-output2}.
\end{proof}

We next choose $N$ so that the contribution of $E_{2N}$ to
$\|{\bf f}_N\|_{L^p(W_N)}$ is small and the sum used to estimate
$\|(\mathbb S-\mathbb S^*){\bf f}_N\|_{L^p(W_N)}$ has the required size.
\begin{lemma}\label{choice}
Fix $p\ge2$, let $Q$ be sufficiently large, and let $r_b$ be as in
\eqref{eq:rb}. Put $\rho=r_b/2$ and $\lambda=-\log\rho$.
Every integer
\begin{equation}\label{eq:choice-depth}
N\ge\max\left\{6,
\left\lceil\frac{20p\log Q}{\lambda}\right\rceil\right\}
\end{equation}
satisfies
\[
\rho^{N-1}\le Q^{-10p},\qquad
\sum_{k=4}^{N-2}k^p\rho^k
\asymp_p(1-\rho)^{-p-1}\asymp_p Q^{p+1},
\]
with constants independent of $N$ and $Q$.
\end{lemma}
\begin{proof}
By \eqref{eq:d-tau-asymptotics},
$\lambda\asymp1-\rho\asymp_p Q^{-1}$. From
\eqref{eq:choice-depth} and $\rho>1/2$,
\[
\rho^{N-1}=\rho^{-1}e^{-\lambda N}
\le2Q^{-20p}\le Q^{-10p}
\]
for $Q$ sufficiently large.
For the upper bound on the sum, $0<\lambda<1$ gives
\[
\sum_{k=4}^{N-2}k^pe^{-\lambda k}
\le\sum_{k=1}^{\infty}k^pe^{-\lambda k}
\le e^\lambda\int_0^\infty x^pe^{-\lambda x}\,dx
\lesssim_p\lambda^{-p-1}.
\]
For the lower bound, take the integers
$\lceil\lambda^{-1}\rceil\le k\le\lfloor2\lambda^{-1}\rfloor$.
When $Q$ is sufficiently large, these lie between $4$ and $N-2$.
There are at least $c\lambda^{-1}$ such integers, and each satisfies
$k^pe^{-\lambda k}\ge e^{-2}\lambda^{-p}$.
Their sum is therefore at least $c_p\lambda^{-p-1}$.
\end{proof}

To turn the pointwise estimate into a norm lower bound, we use
the duality of the weighted spaces. Let $q=p/(p-1)$.
Under the unweighted integral pairing, the dual of $L^p(W)$ is
$L^q(W^{-q/p})$. 
We use this duality with an explicit test function.

\begin{proof}[Proof of Theorem \ref{thm:main} for $p\ge2$]
Fix $p\ge2$. Choose $Q$ sufficiently large, depending only on $p$,
so that all preceding lemmas apply and $\ell\ge19/10$.
Set $\delta_0=(2-\ell)/4$, and then
choose a finite integer $N$ satisfying Lemma~\ref{choice}.
All preceding implicit constants are uniform in $N$, $Q$, and the
permitted range of $\delta_0$. 
We choose ${\bf g}_N$ along ${\bf b}_L$, the direction estimated in
Lemma~\ref{lem:pathwise-antisymmetric}. The scalar factor below makes
the pairing lower bound and the $q$th power of the dual norm depend
on the same quantity $A_N$. Define
\begin{equation*}
{\bf g}_N(x)=\sum_{k=4}^{N-2}\sum_{L\in E_{2k}}
(k\delta_0)^{p-1}x_{b,k}^{\#}{\bf b}_L\one_L(x),
\end{equation*}
and set
\[
A_N=\sum_{k=4}^{N-2}2^{-k}(k\delta_0)^p x_{b,k}^{\#}>0.
\]
On $L\in E_{2k}$, the spectral formula
\eqref{eq:truncation-pointwise-spectral-form} gives
$W_N^{-1/p}{\bf b}_L=(x_{b,k}^{\#})^{-1/p}{\bf b}_L$.
Since $(p-1)q=p$ and $|\bigcup_{L\in E_{2k}}L|=2^{-k}$,
\begin{align*}
\|{\bf g}_N\|_{L^q(W_N^{-q/p})}^q
&=\sum_{k=4}^{N-2}\sum_{L\in E_{2k}}|L|
\left[(k\delta_0)^{p-1}(x_{b,k}^{\#})^{1-1/p}\right]^q=A_N.
\end{align*}
In particular, ${\bf g}_N$ belongs to the dual space.
The signed bound in Lemma~\ref{lem:pathwise-antisymmetric} gives
\[
\int_{I_0}\langle-(\mathbb S-\mathbb S^*){\bf f}_N(x),{\bf g}_N(x)\rangle\,dx
\ge c_p A_N.
\]
By the duality between $L^p(W_N)$ and $L^q(W_N^{-q/p})$,
\begin{align*}
\|(\mathbb S-\mathbb S^*){\bf f}_N\|_{L^p(W_N)}
&\ge
\frac{\left|\int_{I_0}
\langle-(\mathbb S-\mathbb S^*){\bf f}_N(x),{\bf g}_N(x)\rangle\,dx\right|}
{\|{\bf g}_N\|_{L^q(W_N^{-q/p})}}\\
&\ge c_p A_N^{1/p}.
\end{align*}
Using $x_{b,k}^{\#}=(2-r_b)x_{b,0}r_b^{k-1}$, we compute
\begin{align*}
A_N
&=x_{b,0}\delta_0^p\frac{2-r_b}{r_b}
\sum_{k=4}^{N-2}k^p(r_b/2)^k
\asymp_p x_{b,0}\delta_0^p Q^p=x_{a,0}Q^p,
\end{align*}
where we used Lemma~\ref{choice}, $2-r_b\asymp_p Q^{-1}$,
and $x_{b,0}\delta_0^p=x_{a,0}$. Hence
\begin{equation}
\|(\mathbb S-\mathbb S^*){\bf f}_N\|_{L^p(W_N)}^p
\gtrsim_p x_{a,0}Q^p.
\label{eq:output-pth-power}
\end{equation}

By \eqref{eq:input-norm-bound} and Lemma~\ref{choice},
\begin{equation*}
\|{\bf f}_N\|_{L^p(W_N)}^p
\lesssim_p x_{a,0}Q^{-1/(p-1)}.
\end{equation*}
Dividing \eqref{eq:output-pth-power} by this inequality and taking
powers with exponent $1/p$ gives
\begin{equation*}
\frac{\|(\mathbb S-\mathbb S^*){\bf f}_N\|_{L^p(W_N)}}
{\|{\bf f}_N\|_{L^p(W_N)}}
\gtrsim_p
Q^{(p+1/(p-1))/p}
=
Q^{\,1+1/[p(p-1)]}.
\end{equation*}
Write $\alpha(s)=1+1/[s(s-1)]$ for $s>1$. For the low-exponent
argument and the transfer in Section~\ref{sec:hilbert-lower}, define
the normalized pair
\[
\varphi=\frac{{\bf f}_N}{\|{\bf f}_N\|_{L^p(W_N)}},\qquad
\psi=\frac{{\bf g}_N}{\|{\bf g}_N\|_{L^q(W_N^{-q/p})}}.
\]
Both functions are finite dyadic step functions. The signed pairing
above, together with $A_N\asymp_p x_{a,0}Q^p$, proves
\begin{equation}\label{eq:finite-pair-high}
\|\varphi\|_{L^p(W_N)}=\|\psi\|_{L^q(W_N^{-q/p})}=1,
\quad \left|\int_{I_0}\langle T\varphi(x),\psi(x)\rangle\,dx\right|\gtrsim_p Q^{\alpha(p)}.
\end{equation}
Together with Proposition~\ref{lem:dyadic-Ap-upper}, this proves
the theorem for $p\ge2$, with $W=W_N$ and $f={\bf f}_N$.
\end{proof}

\begin{proof}[Proof of Theorem \ref{thm:main} for $1<p<2$]
Fix the target exponent $p\in(1,2)$ and target parameter $Q$.
Let $q=p/(p-1)>2$ and $R=Q^{1/(p-1)}$. Apply the already proved
case at exponent $q$, with construction parameter $R$, to obtain
a finite step weight $U$ such that
\[
[U]_{A_q}^{\mathrm{dy}}\le C_pQ^{1/(p-1)},
\qquad \left|\int_{I_0}\langle(\mathbb S-\mathbb S^*)\varphi(x),\psi(x)\rangle\,dx\right|\ge c_pQ^{\frac{\alpha(q)}{p-1}}.
\]
The finite pair in
\eqref{eq:finite-pair-high} satisfies
$\|\varphi\|_{L^q(U)}=\|\psi\|_{L^p(U^{-p/q})}=1$. Define
\[
W=U^{-1/(q-1)}=U^{-(p-1)}.
\]
Thus $W^{1/p}=U^{-1/q}$ and $W^{-1/(p-1)}=U$.

The characteristic duality in
\cite[Lemma~2.4]{cruzuribe-isralowitz-moen-2018} also holds with
the supremum restricted to dyadic intervals. Thus
\[
[W]_{A_p}^{\mathrm{dy}}
\asymp_p\bigl([U]_{A_q}^{\mathrm{dy}}\bigr)^{p/q}
\lesssim_p R^{p/q}=R^{p-1}=Q.
\]

Since $U^{-p/q}=W$ and $W^{-q/p}=U$, we have
\[
\|\psi\|_{L^p(W)}=1,
\qquad
\|\varphi\|_{L^q(W^{-q/p})}=1.
\]
The finite-sum adjoint identity and weighted H\"older inequality give
\[
\begin{aligned}
\|(\mathbb S-\mathbb S^*)\psi\|_{L^p(W)}
&\ge \left|\int_{I_0}\langle(\mathbb S-\mathbb S^*)\psi(x),\varphi(x)\rangle\,dx\right|\\
&=\left|\int_{I_0}\langle(\mathbb S-\mathbb S^*)\varphi(x),\psi(x)\rangle\,dx\right|\\
&\ge c_pQ^{\frac{\alpha(q)}{p-1}}.
\end{aligned}
\]
Thus exchanging the two finite test functions also supplies a normalized
bilinear pair at exponent $p$.
Here the last power is precisely the desired one because
\[
\frac{\alpha(q)}{p-1}
=\frac1{p-1}+\frac{p-1}{p}
=1+\frac1{p(p-1)}=\alpha(p).
\]
Consequently the explicit choice $f=\psi$ proves the required
lower bound $c_pQ^{\alpha(p)}$.
This completes the low-exponent case.
\end{proof}

\section{From Haar shifts to the Hilbert transform}\label{sec:hilbert-lower}

We transfer the dyadic lower bound of Theorem~\ref{thm:main} to the
Hilbert transform by the periodization and remodeling construction of
\cite[Sections 7 to 9]{domelevo-petermichl-treil-volberg-2024}.
For earlier developments of remodeling, see Kakaroumpas and Treil
\cite{kakaroumpas-treil-2021} and the historical discussion there.

\subsection{Periodization, quasi-periodization, and remodeling}
\label{sec:periodization-remodeling}

We now define the three operators precisely.

{
All functions obtained by periodization or quasi-periodization are
initially defined on \(I_0\). When the Hilbert transform on \(\mathbb R\)
or a norm on \(\mathbb R\) is used, these test functions are extended
by zero outside \(I_0\). The remodeled weight, in contrast, is extended
\(1\)-periodically to \(\mathbb R\). Until this extension is invoked,
weighted norms of functions defined on \(I_0\) are taken over \(I_0\).
}

For \(n\ge1\), let \(\operatorname{ch}^n(I)\) denote the \(2^n\) dyadic
subintervals of \(I\) of length \(2^{-n}|I|\), and for intervals
\(I,J\subset\mathbb R\) let \(\psi_{I,J}:J\to I\) be the unique
orientation-preserving affine bijection. For \(N\ge1\), the maps
\(\psi_{I,J}\), \(J\in\operatorname{ch}^N(I)\), glue into the folding map
\(\psi_I^N:I\to I\) with \(\psi_I^N|_J=\psi_{I,J}\). The periodization of
a function \(f\) on \(I\) is
\begin{equation}
	\mathcal P_I^N f=f\circ\psi_I^N,
	\qquad
	(\mathcal P_I^N f)\big|_J=f\circ\psi_{I,J},
	\quad J\in\operatorname{ch}^N(I),
	\label{eq:periodization}
\end{equation}
that is, \(2^N\) shrunken copies of \(f|_I\). The same formula applies to
a matrix weight \(W\) on \(I\): \((\mathcal P_I^NW)(x)=W(\psi_I^Nx)\).

For \(\vec N=(N_1,N_2,\ldots)\) with \(N_k\ge2\), the iterated
periodization \(\mathcal P^{\vec N}\) is defined recursively. The
intervals on which the construction is performed are called
\emph{starting intervals}, and the pieces into which a starting interval
is split are called its \emph{stopping intervals}. The construction
starts from
\[
\operatorname{Start}_1=\{I_{0,-},I_{0,+}\}
\]
and proceeds in orders: given \(\operatorname{Start}_k\), one applies
\(\mathcal P_I^{N_k}\) on every \(I\in\operatorname{Start}_k\), and the
four dyadic grandchildren \(\operatorname{ch}^2(R)\) of every stopping
interval \(R\in\operatorname{ch}^{N_k}(I)\) form the starting intervals
of the next order,
\[
\operatorname{Start}_{k+1}
:=
\bigcup_{I\in\operatorname{Start}_k}\
\bigcup_{R\in\operatorname{ch}^{N_k}(I)}
\operatorname{ch}^2(R).
\]
Since \(f\) and \(W\) are finite dyadic step functions, the construction
stabilizes after finitely many orders; only finitely many of the
frequencies \(N_k\) therefore matter.

In a single step, applied to a starting interval \(K\) with frequency
\(n\ge2\), the stopping intervals are listed from left to right as
\[
\operatorname{ch}^{n}(K)=\{R_1,\ldots,R_{2^n}\}.
\]
The two stopping intervals \(R_1\) and \(R_{2^n}\) touching the endpoints
of \(K\) are called \emph{exceptional}, and the remaining ones are called
\emph{regular}. For the ordinary periodization \(\mathcal P^{\vec N}\) this
distinction is irrelevant, since all stopping intervals are treated in
the same way; it becomes essential for the quasi-periodization
\(\mathcal QP^{\vec N}\), which is defined in terms of it below.

The transfer uses three operators, all built from dyadic rescalings of a
single interval.
\begin{itemize}
	\item \emph{Periodization} \(\mathcal P^{\vec N}\) replaces a function or
	a weight by many smaller copies on fine dyadic subintervals of \(I_0\).

	This rearrangement preserves measure. Therefore, if we periodize the
	function and the weight in the same way, the weighted norm does not
	change (Lemma~\ref{lem:periodization}).
	\item {\emph{Quasi-periodization} \(\mathcal QP^{\vec N}\)
uses the following local rule. For the current function \(u\) on a
starting interval \(K\), with frequency \(n\ge2\), set
\[
(\mathcal QP_K^n u)|_R=
\begin{cases}
u\circ\psi_{K,R},&R\text{ regular},\\
\langle u\rangle_K,&R\text{ exceptional},
\end{cases}
\qquad R\in\operatorname{ch}^n(K).
\]
The average is taken over the whole starting interval \(K\).
The exceptional intervals remain constant; the next starting intervals
are the four grandchildren of each regular interval.
}
	
	The
	union of all exceptional intervals is denoted by \(E_{\vec N}\), and its
	measure tends to zero as the frequencies tend to infinity. The test
	functions in the transfer identity (Lemma~\ref{lem:transfer}) are obtained
	by applying \(\mathcal QP^{\vec N}\).
	\item {\emph{Remodeling} \(\varrho^{\vec N}\) starts with
quasi-periodization of \(W\).} 
	On an exceptional interval, quasi-periodization replaces a weight by
	its average. Averaging does not in general preserve the relation
	\(V=W^{-1/(p-1)}\). We therefore retain the distribution of the source
	weight on these intervals and continue the construction there.

	On each exceptional interval \(I\)
	with source interval \(J\), replace the constant average by the copy
	\(W\circ\psi_{J,I}\). Take \(I\) itself as a new starting interval,
	with source \(J\), and apply the local quasi-periodization step with
	frequency \(2\). The two exceptional stopping intervals retain the
	source \(J\). Each regular stopping interval has four grandchildren,
	which receive the four grandchildren of \(J\) as their sources in
	left-to-right order. Repeat this local procedure on both the new
	exceptional intervals and the regular grandchildren, always with
	frequency \(2\). Stop on an interval once its source weight is constant.
	
	The pointwise iteration defines the remodeled weight
	\(\widetilde W:=\varrho^{\vec N}W\). The distributional construction
	below produces a measurable limit function \(A\), which will be shown to
	agree almost everywhere with \(\widetilde W\). Applying the same
	interval and source assignments to \(V=W^{-1/(p-1)}\) gives
	\(\widetilde V=\widetilde W^{-1/(p-1)}\) almost everywhere, since the
	limit selects the same original matrix value for both weights.
	
	In the limit, the remodeled weight has the correct distribution at every scale, while its \(A_p\) characteristic remains uniformly controlled, as shown by
	Lemma~\ref{lem:weight-extension} and
	Proposition~\ref{cor:pointwise-distributional-correspondence}.
\end{itemize}

The test functions are quasi-periodized while the weight is
remodeled, and the two procedures differ only on the exceptional
intervals, whose total measure is negligible.

Choose \(r\) so that \(f,g,W\) are constant on
\(\mathcal D_{2r+1}\), and let \(E_{\vec N}\) be the union of the
exceptional stopping intervals in the first \(r\) orders of
quasi-periodization. Then
\begin{equation}
	|E_{\vec N}|\le\sum_{k=1}^r2^{1-N_k},
	\qquad
	\mathcal QP^{\vec N}f=\mathcal P^{\vec N}f
	\quad\text{on }I_0\setminus E_{\vec N},
	\label{eq:exceptional-set}
\end{equation}
and the same equality holds for \(g\) and \(W\). Indeed, 
at order \(k\), a starting interval splits into \(2^{N_k}\) stopping
intervals of length \(2^{-N_k}|I|\), of which the two touching the
endpoints of \(I\) are exceptional, so the exceptional fraction is
\(2^{1-N_k}\). The starting intervals at order \(k\) are pairwise
disjoint subsets of \(I_0\), of total length at most \(1\); summing
over \(k=1,\dots,r\) gives the bound on \(|E_{\vec N}|\). No
exceptional interval beyond order \(r\) has to be added to
\(E_{\vec N}\): there the source weight \(W\) is constant on every
source interval, so the average inserted by quasi-periodization
coincides with the periodized value.
The same holds for \(f\) and \(g\).

\begin{lemma}\label{lem:periodization}
	Let \(f,g\) be finite dyadic step functions on \(I_0\), and let \(W\)
	be a positive definite matrix weight on \(I_0\) that is also a finite
	dyadic step function. Then ordinary simultaneous periodization
	preserves weighted norms:
	\begin{equation}
		\begin{aligned}
			\|\mathcal P^{\vec N}f\|_{L^p(\mathcal P^{\vec N}W)}
			&=\|f\|_{L^p(W)},\\
			\|\mathcal P^{\vec N}g\|_{L^q((\mathcal P^{\vec N}W)^{-q/p})}
			&=\|g\|_{L^q(W^{-q/p})}.
		\end{aligned}
		\label{eq:simultaneous-periodization}
	\end{equation}
	
\end{lemma}

\begin{proof}
	On each \(J\in\operatorname{ch}^N(I)\) the bijection \(\psi_{I,J}\)
	has constant Jacobian \(2^N\), and there are \(2^N\) such pieces, so
	\(\psi_I^N\) is measure preserving. Since
	\[
	\bigl|(\mathcal P_I^NW)^{1/p}(h\circ\psi_I^N)\bigr|^p
	=\bigl(|W^{1/p}h|^p\bigr)\circ\psi_I^N
	\]
	for every vector-valued \(h\), weighted integrals are preserved, and
	iteration gives \eqref{eq:simultaneous-periodization}.
\end{proof}

\subsection{Strong $A_p$ condition}
This subsection is devoted to proving that the remodeled weight
\(\widetilde W=\varrho^{\vec N}W\) satisfies the strong matrix \(A_p\)
condition, with a constant controlled by the dyadic characteristic of
\(W\).

Let \(\mathcal A\) be the finite set of positive definite matrices
attained by \(W\). For each dyadic interval \(J\in\mathcal D\), define
the probability measure \(\mu_J\) on \(\mathcal A\) by
\[
\mu_J(\{A\})
=\frac{|\{x\in J:W(x)=A\}|}{|J|},
\qquad A\in\mathcal A.
\]
Equivalently,
\[
\mu_J=\frac1{|J|}\int_J\delta_{W(x)}\,dx,
\]
where \(\delta_A\) is the Dirac probability measure at \(A\). This
use of \(\delta_A\) is distinct from the scalar parameters
\(\delta_n\) in the construction.
For a probability measure \(\mu\) on \(\mathcal A\), put
\begin{equation}
	\widetilde \Phi_p(\mu)
	=\int_{\mathcal A}
	\left(\int_{\mathcal A}
	\|A^{1/p}B^{-1/p}\|_{\op}^{q}\,d\mu(B)\right)^{p-1}
	d\mu(A).
	\label{eq:distribution-functional}
\end{equation}
Then for every \(J\in\mathcal D\)
\[
\widetilde \Phi_p(\mu_J)=\Phi_{p,J}(W).
\]
{
For probability measures \(\rho,\mu\) on \(\mathcal A\), if
\(\rho\le C\mu\), then
\begin{equation}
\widetilde \Phi_p(\rho)\le C^p\widetilde \Phi_p(\mu).
\label{eq:distribution-domination}
\end{equation}
}
{The inner integral increases} by at most a factor \(C\) and the outer integral by at
most a factor \(C\), so the total factor is \(C^{p-1}C=C^p\). Moreover,
the four dyadic {grandchildren} of \(J\) partition \(J\), so
\begin{equation}
	\mu_J=\frac14\sum_{J'\in\operatorname{ch}^2(J)}\mu_{J'},
	\qquad
	\mu_{J'}\le4\mu_J
	\quad(J'\in\operatorname{ch}^2(J)).
	\label{eq:grandchild-laws}
\end{equation}

Hence, for the \(A_p\) estimate we recast the construction in distributional
form: the local \(A_p\) expression depends only on the distribution of
the matrix values on an interval, not on their pointwise arrangement, so
instead of tracking the weight \(W\) itself we track the probability
measure \(\mu_J\) of \(W\) on the associated interval \(J\).
For the sake of clarity, it is convenient to split the quite long reasoning into several intermediate parts.

\medskip
{
\emph{Step 1: construction of the measure-valued function \(\nu\).}
Start with
\[
J^{(0)}(x)=I_{0,\pm}\quad(x\in I_{0,\pm}),
\qquad
\nu^{(m)}(x):=\mu_{J^{(m)}(x)}\quad(m\ge0).
\]
Treat starting intervals in their order of creation, so that parents
come first and every interval with a nonconstant source is eventually
treated. {A starting interval \(S\) with source \(J_S\)
is frozen as soon as \(W|_{J_S}\equiv A_S\). If \(m_S\) is its
creation stage, then
\[
\mu_{J_S}=\delta_{A_S},\qquad
\nu^{(m)}|_S=\delta_{A_S}\quad(m\ge m_S).
\]
This includes the initial half-intervals \((m_S=0)\), exceptional
intervals, and regular grandchildren.
}

At stage \(m\), let \(K_m\) have source \(J_m\). As in the
remodeling construction, use
\[
n_m=
\begin{cases}
N_k,&\text{on an original branch of order }k,\\
2,&\text{after entering an exceptional interval}.
\end{cases}
\]
List the stopping intervals and all grandchildren from left to right:
\[
\begin{gathered}
\operatorname{ch}^{n_m}(K_m)=\{R_1,\ldots,R_{2^{n_m}}\},
\qquad K_m^{\mathrm{ex}}=R_1\cup R_{2^{n_m}},\\
\operatorname{ch}^2(J_m)=\{J_m^{(1)},\ldots,J_m^{(4)}\},
\qquad
\operatorname{ch}^2(R)=\{R^{(1)},\ldots,R^{(4)}\}.
\end{gathered}
\]
The source labels are updated by
\[
J^{(m+1)}(x)=
\begin{cases}
J^{(m)}(x),&x\notin K_m,\\
J_m,&x\in K_m^{\mathrm{ex}},\\
J_m^{(\ell)},&x\in R^{(\ell)},\quad R\text{ regular},\quad1\le\ell\le4.
\end{cases}
\]
The exceptional intervals and regular grandchildren are the new
starting intervals. If none remain to be treated, set
\(J^{(m+1)}=J^{(m)}\). In particular,
\[
\nu^{(m)}|_{K_m}=\mu_{J_m},
\qquad
\nu^{(m+1)}(x)=\mu_{J^{(m+1)}(x)}.
\]

For \(0<|L|<\infty\), define the average law by
\[
\begin{aligned}
\nu_L^{(m)}&:=\frac1{|L|}\int_L\nu^{(m)}(x)\,dx,\\
\nu_L^{(m)}(\{A\})
&=\frac1{|L|}\int_L\mu_{J^{(m)}(x)}(\{A\})\,dx
\qquad(A\in\mathcal A).
\end{aligned}
\]
}

Extend every intermediate law periodically to the real line. Write
\[
\mathcal D_{\R}
=\bigl\{[j2^{-m},(j+1)2^{-m}):
\ j,m\in\mathbb Z\bigr\},
\]
and let \(\mathcal D_{\mathrm{sd}}\) be the collection of all unions of
two adjacent members of \(\mathcal D_{\R}\) of equal length; the two
members need not have the same dyadic parent.

{
\begin{lemma}\label{ex-W}
Let \(W\) be a positive definite finite dyadic step weight on \(I_0\),
and let \(\nu^{(m)}\) be the laws constructed above.
\begin{enumerate}[\rm (i)]
\item If \(S\) is a frozen starting interval created by
stage \(m\), with source \(J_S\) satisfying \(W|_{J_S}\equiv A_S\), then
\[
\nu^{(m)}|_S=\delta_{A_S},\qquad
\mathcal P^{\vec N}W=\mathcal QP^{\vec N}W=A_S
\quad\text{a.e. on }S\setminus E_{\vec N}.
\]
\item There is a measurable \(A:I_0\to\mathcal A\) such that, for
almost every \(x\in I_0\),
\[
\exists M(x)<\infty:\qquad
\nu^{(m)}(x)=\delta_{A(x)}\quad(m\ge M(x)).
\]
In particular, \(\nu^{(m)}(x)\to\delta_{A(x)}\) almost everywhere.
\end{enumerate}
\end{lemma}

\begin{proof}
{
(i) The frozen source on \(S\) remains \(J_S\), so
\[
\nu^{(m)}|_S=\mu_{J_S}=\delta_{A_S}.
\]
Outside \(E_{\vec N}\), both constructions start with sources
\(I_{0,\pm}\) and assign the same four source grandchildren to each
regular interval. Their source labels therefore agree until freezing.
Since \(W|_{J_S}\equiv A_S\), all further ordinary periodizations
remain constant there. Together with \eqref{eq:exceptional-set}, this gives
\[
\mathcal P^{\vec N}W=\mathcal QP^{\vec N}W=A_S
\quad\text{a.e. on }S\setminus E_{\vec N}.
\]
}

{(ii)} For a starting interval \(R\) with nonconstant source,
let \(E_s(R)\) consist of points undergoing \(s\) consecutive
exceptional refinements from \(R\). Since all frequencies are at least
\(2\),
\[
E_{s+1}(R)\subseteq E_s(R),\qquad
|E_{s+1}(R)|\le\tfrac12|E_s(R)|,
\qquad |E_s(R)|\le2^{-s}|R|.
\]
Let \(Z\) be the set of interval endpoints. There are only countably
many starting intervals and endpoints, so
\[
B:=Z\cup
\bigcup_{R:\,W|_{J_R}\text{ is nonconstant}}
\bigcap_{s\ge1}E_s(R),
\qquad |B|=0.
\]

Choose \(r\) with \(W\) constant on \(\mathcal D_{2r+1}\), and let
\(a_m(x)\) count the regular refinements along the chain of \(x\)
by stage \(m\). Then
\[
|J^{(m)}(x)|=2^{-1}4^{-a_m(x)}.
\]
If \(a_m(x)\ge r\), then \(|J^{(m)}(x)|\le2^{-(2r+1)}\),
which implies that \(W|_{J^{(m)}(x)}\) is constant.
Fix \(x\notin B\). If its chain never freezes, the ordering in Step~1
and the exclusion of infinite exceptional chains give
$
a_m(x)\to\infty$. Thus, there exists $m$ such that
$ a_m(x)\ge r,$
contradicting the stopping rule. Let \(M(x)\) be the first stage with
constant source weight, and \(A(x)\) its value. Freezing gives
\[
W|_{J^{(M(x))}(x)}\equiv A(x),
\qquad
\nu^{(m)}(x)=\delta_{A(x)}\quad(m\ge M(x)).
\]

Define \(A\) arbitrarily on \(B\). For each \(A_0\in\mathcal A\),
\[
\mathbf1_{\{A(x)=A_0\}}
=\lim_{m\to\infty}\nu^{(m)}(x)(\{A_0\})
\quad\text{a.e.}
\]
The functions on the right are measurable; hence \(A\) is measurable.
\end{proof}
}

\medskip
\emph{Step 2: estimate at finite stages}.
We establish the bound
\begin{equation}
	\widetilde\Phi_p(\nu_L^{(m)})
	\leq 4^p[W]_{A_p}^{\mathrm{dy}}
	\qquad
	(m\geq0,\ L\in\mathcal D_{\mathrm{sd}})
	\label{eq:law-invariant}
\end{equation}
for the averaged laws at every finite stage, by induction on \(m\).
\begin{lemma}\label{lem:weight-extension1}
	For any fixed  $m,$
	\[
	\widetilde\Phi_p(\nu_{L}^{(m)})
	\leq4^p[W]_{A_p}^{\mathrm{dy}}
	\qquad
	(L\in\mathcal D_{\mathrm{sd}}).
	\]
\end{lemma}
\begin{proof}
	
	\emph{The case \(m=0\).}
	At the initial stage,
	\[
	\mu_{I_0}
	=
	\frac12\bigl(\mu_{I_{0,+}}+\mu_{I_{0,-}}\bigr),
	\qquad
	\mu_{I_{0,\pm}}\leq2\mu_{I_0}.
	\]
	Since \(\nu^{(0)}\) takes only the values
	\(\mu_{I_{0,+}}\) and \(\mu_{I_{0,-}}\), we have
	\[
	\nu_L^{(0)}
	=
	\frac1{|L|}\int_L\nu^{(0)}(x)\,dx
	\leq2\mu_{I_0}.
	\]
	Therefore, 
	\[
	\begin{aligned}
		\widetilde\Phi_p(\nu_L^{(0)})
		&\leq2^p\widetilde\Phi_p(\mu_{I_0})  =2^p\Phi_{p,I_0}(W)
		\leq2^p[W]_{A_p}^{\mathrm{dy}}.
	\end{aligned}
	\]

Assume that \eqref{eq:law-invariant} holds at stage \(m\).
If all source intervals are frozen, there is nothing to prove.
Otherwise, the next periodic stage updates every translate
\(K_m+j\), \(j\in\mathbb Z\). We first prove that one local update
preserves the bound for all \(L\in\mathcal D_{\mathrm{sd}}\).

Let \(K=K_m+j\) have source \(J=J_m\) and frequency \(n\ge2\).
Let \(\eta\) be the law before this local update, so that
\[
\eta|_K=\mu_J,
\qquad
\widetilde\Phi_p(\eta_L)\le4^p[W]_{A_p}^{\mathrm{dy}}
\quad(L\in\mathcal D_{\mathrm{sd}}).
\]
List the grandchildren \(J^{(1)},\ldots,J^{(4)}\) of \(J\) and,
for each regular stopping interval \(R\), its grandchildren
\(R^{(1)},\ldots,R^{(4)}\), in left-to-right order. The updated law is
\[
\eta'(x)=
\begin{cases}
\eta(x),&x\notin K,\\
\mu_J,&x\text{ belongs to an exceptional stopping interval},\\
\mu_{J^{(\ell)}},&x\in R^{(\ell)}\text{ for a regular stopping interval }R.
\end{cases}
\]
By \eqref{eq:grandchild-laws},
\begin{equation}\label{eq:single-update-domination}
\eta'(x)\le4\mu_J\quad\text{for a.e. }x\in K.
\end{equation}
Also, the average law on every stopping interval is unchanged:
\[
\eta'_R=\mu_J\quad(R\in\operatorname{ch}^{n}(K)),
\qquad
\int_K\eta'(x)\,dx=|K|\mu_J=\int_K\eta(x)\,dx.
\]
Fix \(L\in\mathcal D_{\mathrm{sd}}\). We distinguish four cases,
with intersections understood up to endpoints.

\emph{Case 1: \(L\subseteq K\).}
Averaging \eqref{eq:single-update-domination} gives
\(\eta'_L\le4\mu_J\). Therefore
\[
\widetilde\Phi_p(\eta'_L)
\le4^p\widetilde\Phi_p(\mu_J)
=4^p\Phi_{p,J}(W)
\le4^p[W]_{A_p}^{\mathrm{dy}}.
\]

\emph{Case 2: \(K\subsetneq L\).}
Since \(\eta'=\eta\) outside \(K\) and the total law on \(K\)
is preserved,
\[
|L|(\eta'_L-\eta_L)=\int_K(\eta'-\eta)\,dx=0.
\]
Thus \(\eta'_L=\eta_L\), and the bound follows from the hypothesis
on \(\eta\).

\emph{Case 3: \(K\cap L=\varnothing\).}
The law is unchanged on \(L\), so the same hypothesis applies.

\emph{Case 4: \(K\cap L\ne\varnothing\), \(K\nsubseteq L\), and
\(L\nsubseteq K\).}
Write \(L=I_1\cup I_2\), where \(I_1,I_2\) are adjacent dyadic
intervals of equal length. Dyadic intervals in the same grid are
either disjoint or nested. If \(|I_1|\ge|K|\), any nonempty
intersection with \(K\) would imply \(K\subseteq I_i\subseteq L\),
which is excluded. Hence \(|I_1|=|I_2|<|K|\), and each \(I_i\)
meeting \(K\) is contained in \(K\). Exactly one meets \(K\), since
\(L\nsubseteq K\). Consequently \(D:=L\cap K\) is that dyadic
interval, and adjacency forces it to share an endpoint with \(K\).

Let \(s=2^{-n}|K|\) be the stopping length. If \(|D|\le s\),
dyadic nesting places \(D\) inside the exceptional stopping interval
at that endpoint, where \(\eta'=\mu_J\). If \(|D|>s\), the stopping
grid partitions \(D\) into complete stopping intervals, each with
average law \(\mu_J\). In either case,
$
\eta'_D=\mu_J=\eta_D.
$
Since the law is unchanged on \(L\setminus K\),
\[
|L|(\eta'_L-\eta_L)=|D|(\eta'_D-\eta_D)=0.
\]
The hypothesis on \(\eta\) therefore gives the desired bound.

For a fixed bounded \(L\), only finitely many translates \(K_m+j\)
meet \(L\). They have disjoint interiors, and each still carries the
source law \(\mu_{J_m}\) before its update. Applying the local claim
successively to these translates preserves the bound at every update.
The resulting law agrees with \(\nu^{(m+1)}\) on \(L\). Hence
\[
\widetilde\Phi_p(\nu_L^{(m+1)})\le4^p[W]_{A_p}^{\mathrm{dy}}
\quad(L\in\mathcal D_{\mathrm{sd}}),
\]
which completes the induction.
\end{proof}

\emph{Step 3: passage to the limit.}
\begin{lemma}\label{lem:limit-stage}
	Let \(A:I_0\to\mathcal A\) be the measurable function determined by
	\[
	\nu^{(m)}(x)\longrightarrow\delta_{A(x)}
	\qquad\text{for a.e. }x\in I_0,
	\]
	and extend \(A\) \(1\)-periodically to \(\mathbb R\). Then
	\[
	\Phi_{p,L}(A)
	\le4^p[W]_{A_p}^{\mathrm{dy}}
	\qquad(L\in\mathcal D_{\mathrm{sd}}).
	\]
\end{lemma}
\begin{proof}
	{By the convergence assertion of Lemma~\ref{ex-W},} for almost every \(x\in I_0\) there are
	\(M(x)<\infty\) and \(A(x)\in\mathcal A\) with
	\[
	\nu^{(m)}(x)=\delta_{A(x)}
	\qquad\text{for every }m\geq M(x).
	\]
	Define \(A\) arbitrarily on the exceptional null set and extend it
	\(1\)-periodically to \(\mathbb R\). Then
	\[
	\nu^{(m)}(x)
	\longrightarrow
	\delta_{A(x)}
	\qquad\text{for a.e. }x\in\mathbb R.
	\]
	Equivalently, for every \(C\in\mathcal A\) the densities
	\(\nu^{(m)}(x)(\{C\})\) lie in \([0,1]\) and converge almost everywhere
	to \(\mathbf 1_{\{A(x)=C\}}\).
	
	Fix \(L\in\mathcal D_{\mathrm{sd}}\). Since \(\mathcal A\) is finite,
	dominated convergence gives
	\[
	\begin{aligned}
		\nu_L^{(m)}(\{C\})
		&=
		\frac1{|L|}
		\int_L\nu^{(m)}(x)(\{C\})\,dx\longrightarrow
		\frac1{|L|}
		\int_L
		\mathbf 1_{\{A(x)=C\}}\,dx=
		\mu_L^A(\{C\}),
	\end{aligned}
	\]
	where
	\[
	\mu_L^A
	:=
	\frac1{|L|}
	\int_L\delta_{A(x)}\,dx.
	\]
	Write
	\[
	\mathcal A=\{A_1,\ldots,A_M\}
	\]
	and set
	\[
	a_i^{(m)}
	:=
	\nu_L^{(m)}(\{A_i\}),
	\qquad
	a_i
	:=
	\mu_L^A(\{A_i\}),
	\qquad
	c_{ij}
	:=
	\|A_i^{1/p}A_j^{-1/p}\|_{\op}^{q}.
	\]
	Then \(a_i^{(m)}\to a_i\), and every \(c_{ij}\) is finite. Since
	\(p-1>0\), the expression below is a continuous function of the
	finitely many variables \(a_i^{(m)}\), so
	{
\[
\begin{aligned}
\widetilde\Phi_p(\nu_L^{(m)})
&=\sum_{i=1}^{M}a_i^{(m)}
\left(\sum_{j=1}^{M}c_{ij}a_j^{(m)}\right)^{p-1}\\
&\longrightarrow
\sum_{i=1}^{M}a_i
\left(\sum_{j=1}^{M}c_{ij}a_j\right)^{p-1}
=\widetilde\Phi_p(\mu_L^A).
\end{aligned}
\]
}
	Passing to the limit in \eqref{eq:law-invariant}, we obtain
	\[
	\Phi_{p,L}(A)
	=
	\widetilde\Phi_p(\mu_L^A)
	\leq4^p[W]_{A_p}^{\mathrm{dy}}
	\qquad
	(L\in\mathcal D_{\mathrm{sd}}),
	\] which completes the proof.
\end{proof}

\begin{lemma}\label{lem:weight-extension}
	Let \(A\) be the distributional limit function of
	Lemma~\ref{lem:limit-stage}. Then
	\[
	[A]_{A_p}
	\le16^p[W]_{A_p}^{\mathrm{dy}}.
	\]
\end{lemma}

\begin{proof}
	By Lemma~\ref{lem:limit-stage},
	\[
	\Phi_{p,L}(A)\le4^p[W]_{A_p}^{\mathrm{dy}}
	\qquad(L\in\mathcal D_{\mathrm{sd}}).
	\]
	It remains to replace the dyadic unions \(L\in\mathcal D_{\mathrm{sd}}\)
	by arbitrary intervals, which costs one further factor \(4^p\); together
	with the factor \(4^p\) just displayed this gives the factor \(16^p\) of
	the statement.
	
	Let \(B\subset\mathbb R\) be an interval with
	\(0<|B|<\infty\), and choose the dyadic length \(h\) with
	\[
	|B|\le h<2|B|.
	\]
	Since \(|B|\le h\), the interval \(B\) meets at most two adjacent
	dyadic intervals of length \(h\); let \(L\) be their union, enlarged by
	one adjacent dyadic interval of length \(h\) if \(B\) meets only one.
	Consequently,
	\[
	L\in\mathcal D_{\mathrm{sd}},
	\qquad
	B\subseteq L,
	\qquad
	|L|=2h\leq4|B|.
	\]
	It follows that
	\[
	\mu_B^A
	\leq
	\frac{|L|}{|B|}
	\mu_L^A
	\leq4\mu_L^A.
	\]
	Applying \eqref{eq:distribution-domination} gives
	\[
	\begin{aligned}
		\Phi_{p,B}(A)
		&=
		\widetilde\Phi_p(\mu_B^A)\leq
		4^p\widetilde\Phi_p(\mu_L^A)=
		4^p\Phi_{p,L}(A)
		\leq
		16^p[W]_{A_p}^{\mathrm{dy}}.
	\end{aligned}
	\]
	Taking the supremum over all intervals \(B\) yields
	\[
	[A]_{A_p}
	\leq16^p[W]_{A_p}^{\mathrm{dy}}.
	\]
	Here the factor \(16^p\) is the product of the factor \(4^p\) from
	\eqref{eq:law-invariant} and the factor \(4^p\) from the enlargement
	\(B\mapsto L\).
\end{proof}

{
We next show that the distributional limit agrees with the pointwise
remodeled weight.
}

\begin{proposition}
	\label{cor:pointwise-distributional-correspondence}
	Let \(\widetilde W:=\varrho^{\vec N}W\) be the pointwise remodeled
	weight, and let \(A\) be the measurable function determined by the
	distributional limit. Then
$
	\widetilde W=A$
a.e. on $I_0.$
Moreover,
\[
\widetilde W=A=\mathcal P^{\vec N}W=\mathcal QP^{\vec N}W
\quad\text{a.e. on }I_0\setminus E_{\vec N}.
\]
Consequently,
$
	[\widetilde W]_{A_p}
	\leq16^p[W]_{A_p}^{\mathrm{dy}}.
$
\end{proposition}

\begin{proof}
By Lemma~\ref{ex-W}(ii), the frozen starting intervals cover \(I_0\)
up to a null set. Let \(S\) be one of them, with source \(J_S\),
value \(A_S\), and creation stage \(m_S\ge0\). Then
\[
W|_{J_S}\equiv A_S,\qquad
\nu^{(m)}|_S=\delta_{A_S}\quad(m\ge m_S),\qquad
A|_S=A_S\quad\text{a.e.}
\]
The pointwise construction uses the same interval and source
assignments. On \(S\), its source copy is already constant:
\[
W\circ\psi_{J_S,S}=A_S\quad\text{a.e. on }S.
\]
It therefore freezes at the same value, including when \(S=I_{0,\pm}\)
and \(m_S=0\). There are only countably many such intervals, so
\[
\widetilde W=A\quad\text{a.e. on }I_0.
\]
Applying Lemma~\ref{ex-W}(i) on these intervals also gives
\[
\widetilde W=A=\mathcal P^{\vec N}W=\mathcal QP^{\vec N}W
\quad\text{a.e. on }I_0\setminus E_{\vec N}.
\]
Finally, Lemma~\ref{lem:weight-extension} yields
$$
[\widetilde W]_{A_p}=[A]_{A_p}
\le16^p[W]_{A_p}^{\mathrm{dy}},$$
which completes the proof.
\end{proof}

{
We next show that the weighted norms of the quasi-periodized test
functions converge to their original norms.
}

\begin{lemma}\label{lem:norm-tracking}
	Let \(W,f,g\) be as in Lemma~\ref{lem:periodization}, let
	\(\widetilde W=\varrho^{\vec N}W\) be the pointwise remodeled weight,
	let \(r\) be the common finite depth determined by \(W,f,g\) (so
	that only the frequencies \(N_1,\dots,N_r\) matter), and use the same
	frequency vector \(\vec N=(N_1,N_2,\dots)\) for \(W\), \(f\), and
	\(g\). Then for every \(\varepsilon>0\) there is
	\(\ell=\ell(W,f,g,\varepsilon)<\infty\) such that every frequency
	vector with
$
	N_k\ge\ell$ $(1\le k\le r)$
	satisfies
	\[
	\left|\|\mathcal QP^{\vec N}f\|_{L^p(\widetilde W)}
	-\|f\|_{L^p(W)}\right|\le\varepsilon,\]
	and 
	\[
	\left|\|\mathcal QP^{\vec N}g\|_{L^q(\widetilde W^{-q/p})}
	-\|g\|_{L^q(W^{-q/p})}\right|\le\varepsilon.
	\]
	In particular both norms converge to the corresponding norms of
	\(f\) and \(g\) along every family of frequency vectors with
	\(\min_{1\le k\le r}N_k\to\infty\).
\end{lemma}

\begin{proof}
Let \(\mathcal A\) be the finite set of values of \(W\). By
Proposition~\ref{cor:pointwise-distributional-correspondence} and the
pullback definition of ordinary periodization,
\[
\widetilde W(x),\ (\mathcal P^{\vec N}W)(x)\in\mathcal A
\quad\text{a.e. on }I_0.
\]
Set
\[
C_W:=\max_{B\in\mathcal A}
\max\{\|B\|_{\op},\|B^{-1}\|_{\op}\}<\infty.
\]
For \(B\in\mathcal A\) and every vector \(u\),
\[
|B^{1/p}u|^p\le C_W|u|^p,\qquad
|B^{-1/p}u|^q\le C_W^{q/p}|u|^q.
\]
Write \(E=E_{\vec N}\). By \eqref{eq:exceptional-set} and
Proposition~\ref{cor:pointwise-distributional-correspondence},
\[
\begin{gathered}
|E|\le\sum_{k=1}^r2^{1-N_k},\qquad
\widetilde W=\mathcal P^{\vec N}W
\quad\text{a.e. on }I_0\setminus E,\\
\mathcal QP^{\vec N}h=\mathcal P^{\vec N}h
\quad\text{on }I_0\setminus E,\qquad h=f,g.
\end{gathered}
\]
Pullbacks and averages also give
\[
\max\{\|\mathcal P^{\vec N}h\|_\infty,
\|\mathcal QP^{\vec N}h\|_\infty\}\le\|h\|_\infty,
\qquad h=f,g.
\]
The weighted densities therefore agree outside \(E\). Using
Lemma~\ref{lem:periodization} and bounding their integrals over \(E\),
we obtain
\[
\begin{aligned}
\left|\|\mathcal QP^{\vec N}f\|_{L^p(\widetilde W)}^p
-\|f\|_{L^p(W)}^p\right|
&\le2C_W\|f\|_\infty^p|E|,\\
\left|\|\mathcal QP^{\vec N}g\|_{L^q(\widetilde W^{-q/p})}^q
-\|g\|_{L^q(W^{-q/p})}^q\right|
&\le2C_W^{q/p}\|g\|_\infty^q|E|.
\end{aligned}
\]
Choose \(\ell\ge2\) so that
\[
2r\,2^{1-\ell}
\max\left\{
\frac{C_W\|f\|_\infty^p}{\varepsilon^p},
\frac{C_W^{q/p}\|g\|_\infty^q}{\varepsilon^q}
\right\}\le1.
\]
For \(N_k\ge\ell\), we have \(|E|\le r2^{1-\ell}\).
The inequality \(|a-b|^s\le|a^s-b^s|\) for \(a,b\ge0\), \(s\ge1\),
then gives the two asserted norm estimates.
\end{proof}

\subsection{The Hilbert transform and the transfer identity}

\leavevmode{
We use the normalizations
\[
\begin{aligned}
Hf(s)&=\frac1\pi\,\mathrm{p.v.}\int_{\R}\frac{f(t)}{s-t}\,dt,\\
H^{\mathbb T}f(s)&=\mathrm{p.v.}\int_0^1f(t)\cot\bigl(\pi(s-t)\bigr)\,dt,
\qquad \mathbb T=\R/\Z,
\end{aligned}
\]
acting componentwise, with the Haar convention
\eqref{eq:haar-normalization}. Following
\cite[Sections~6 and~7.3]{domelevo-petermichl-treil-volberg-2024}, set
\[
\begin{aligned}
c_0&=(H\one_{I_{0,+}},\one_{I_{0,-}})_{L^2(\R)},\\
c_1&=(H^{\mathbb T}h_{I_0},h_{I_{0,+}})_{L^2(I_0)}\ne0.
\end{aligned}
\]
Half-period translation and skew-adjointness give
\[
(H^{\mathbb T}h_{I_{0,+}},h_{I_{0,-}})
=(H^{\mathbb T}h_{I_{0,-}},h_{I_{0,+}})
=-(H^{\mathbb T}h_{I_{0,+}},h_{I_{0,-}})=0.
\]
Thus the dyadic operator in the transfer identity reduces to
\begin{equation}
H^{dy}=c_1(\mathbb S-\mathbb S^*).
\label{eq:hdy}
\end{equation}
We now apply the transfer identity of
\cite[Lemma~7.9]{domelevo-petermichl-treil-volberg-2024}.
}

\begin{lemma}\label{lem:transfer}
	{
Let \(\mathbf f,\mathbf g\) be finite dyadic step functions on
\(I_0\), and choose any \(r\ge1\) such that both functions are constant
on \(\mathcal D_{2r+1}\). Only the frequencies \(N_1,\dots,N_r\)
can affect their quasi-periodizations.
}
	There is a sequence of frequency vectors \(\vec N^{(j)}\), with
	\(N_k^{(j)}\to\infty\) as \(j\to\infty\) for every \(k\le r\),
	such that
	\begin{equation}
		\begin{aligned}
			&(H\mathcal QP^{\vec N^{(j)}}\mathbf f,
			\mathcal QP^{\vec N^{(j)}}\mathbf g)_{L^2(\R)}\\
			&\qquad\longrightarrow
			(H^{dy}\mathbf f,\mathbf g)_{L^2(I_0)}
			+c_0\Bigl[
			(\langle\mathbf f\rangle_{I_{0,+}},
			\langle\mathbf g\rangle_{I_{0,-}})
			-(\langle\mathbf f\rangle_{I_{0,-}},
			\langle\mathbf g\rangle_{I_{0,+}})
			\Bigr].
		\end{aligned}
		\label{eq:transfer-identity}
	\end{equation}
	The inner products in the bracket are Euclidean inner products of
	the averages \(\langle\mathbf f\rangle_{I_{0,\pm}}\) and
	\(\langle\mathbf g\rangle_{I_{0,\pm}}\).
\end{lemma}

{
If \(r\) is larger than the depth needed by \(\mathbf f,\mathbf g\),
choose the unused coordinates of \(\vec N^{(j)}\) to tend to infinity
as well. Those coordinates do not change the quasi-periodized functions,
since all subsequent source functions are constant.

The bracket in \eqref{eq:transfer-identity} comes from the retained
first-level averages. With
\[
\mathbb E_1u
=\mathbf1_{I_{0,+}}\langle u\rangle_{I_{0,+}}
+\mathbf1_{I_{0,-}}\langle u\rangle_{I_{0,-}},
\]
and zero extension outside \(I_0\), skew-adjointness gives
\[
(H\mathbb E_1\mathbf f,\mathbb E_1\mathbf g)_{L^2(\mathbb R)}
=c_0\Bigl[
(\langle\mathbf f\rangle_{I_{0,+}},\langle\mathbf g\rangle_{I_{0,-}})
-(\langle\mathbf f\rangle_{I_{0,-}},\langle\mathbf g\rangle_{I_{0,+}})
\Bigr].
\]
This term persists as the frequencies increase; it does not arise
from the exceptional stopping intervals, whose contribution vanishes
in the limit. It is estimated separately in
Lemma~\ref{lem:boundary-term}.
}

\subsection{Proof of Theorem~\ref{thm:main-1}}

\leavevmode{
For every \(1<p<\infty\), we combine the dyadic lower bound
(Theorem~\ref{thm:main}), the transfer identity
(Lemma~\ref{lem:transfer}), and the uniform \(A_p\) bound for the
remodeled weight
(Proposition~\ref{cor:pointwise-distributional-correspondence})
to obtain the Hilbert transform lower bound by the same argument.
}

{
We first express the dyadic lower bound using normalized test functions.
}

\begin{lemma}\label{lem:pairing-form}
	Let {\(1<p<\infty\)}, and choose the finite dyadic step weight \(W\)
	constructed in the proof of Theorem~\ref{thm:main} for this \(p\)
	and \(Q\). There exist finite dyadic step functions
	\(\mathbf f,\mathbf g\) on \(I_0\) such that
$
	\|\mathbf f\|_{L^p(W)}
	=\|\mathbf g\|_{L^q(W^{-q/p})}
	=1
$
	and
	\[
	\bigl|((\mathbb S-\mathbb S^*)\mathbf f,\mathbf g)_{L^2(I_0)}\bigr|
	\ge c_pQ^{1+1/[p(p-1)]}.
	\]
	Consequently, for the same pair,
	\[
	\bigl|(H^{dy}\mathbf f,\mathbf g)_{L^2(I_0)}\bigr|
	\ge c_pQ^{1+1/[p(p-1)]}.
	\]
\end{lemma}

\begin{proof}
{
For every \(1<p<\infty\), the proof of Theorem~\ref{thm:main}
provides a normalized finite dyadic pair satisfying the first
inequality. Since \(c_1\ne0\), the second inequality follows from
\eqref{eq:hdy}.
}
\end{proof}

We next estimate the boundary term in \eqref{eq:transfer-identity}.

\begin{lemma}\label{lem:boundary-term}
	Let \(W\) be a finite positive definite step weight on \(I_0\)
	such that \([W]_{A_p}^{\mathrm{dy}}<\infty\). If
$
	\|\mathbf f\|_{L^p(W)}
	=\|\mathbf g\|_{L^q(W^{-q/p})}
	=1,
$
	then
	\[
	\left|c_0\Bigl[
	(\langle\mathbf f\rangle_{I_{0,+}},
	\langle\mathbf g\rangle_{I_{0,-}})
	-(\langle\mathbf f\rangle_{I_{0,-}},
	\langle\mathbf g\rangle_{I_{0,+}})
	\Bigr]\right|
	\le8|c_0|\bigl([W]_{A_p}^{\mathrm{dy}}\bigr)^{1/p}.
	\]
\end{lemma}

\begin{proof}
	Let \(\mathbb E_0u=\one_{I_0}\langle u\rangle_{I_0}\).
	Hölder's inequality and \(|I_0|=1\) give
	\[
	\|\mathbb E_0u\|_{L^p(W)}
	\le\bigl([W]_{A_p}^{\mathrm{dy}}\bigr)^{1/p}
	\|u\|_{L^p(W)}.
	\]
	For \(u=\one_{I_{0,+}}\mathbf f\) and
	\(v=\one_{I_{0,-}}\mathbf g\), we have
	\[
	(\langle\mathbf f\rangle_{I_{0,+}},
	\langle\mathbf g\rangle_{I_{0,-}})
	=4\int_{I_0}\langle\mathbb E_0u,v\rangle\,dx.
	\]
	Weighted duality bounds this by
	\(4([W]_{A_p}^{\mathrm{dy}})^{1/p}\). Interchanging the two halves
	gives the same bound for the other term. Summing proves the claim.
\end{proof}

\begin{proof}[{Proof of Theorem~\ref{thm:main-1}}]
	{Fix \(1<p<\infty\) and take \(Q\) sufficiently large.}
	Let \(W\) be the finite dyadic step weight constructed in the proof
	of Theorem~\ref{thm:main}, with
	\[
	[W]_{A_p}^{\mathrm{dy}}\le C_pQ,
	\]
	and let \(\mathbf f,\mathbf g\) be the normalized pair from
	Lemma~\ref{lem:pairing-form}.

{
Choose \(r\ge1\) so that \(W,\mathbf f,\mathbf g\) are constant
on \(\mathcal D_{2r+1}\). Apply Lemma~\ref{lem:transfer} with this \(r\). Thus its sequence satisfies \(N_k^{(j)}\to\infty\) for every
\(1\le k\le r\), including any coordinates unused by the test functions.
}

	The main term in \eqref{eq:transfer-identity} has size at least
	\(c_pQ^{1+1/[p(p-1)]}\), whereas the boundary term is bounded by
	\(C_pQ^{1/p}\). Since
$
	1+\frac1{p(p-1)}>\frac1p,
$
	the boundary term is absorbed into the main lower bound for
	sufficiently large \(Q\). Along the sequence of frequencies in
	Lemma~\ref{lem:transfer}, we therefore have, for sufficiently large
	frequencies,
	\[
	\bigl|(H\widetilde{\mathbf f},\widetilde{\mathbf g})_{L^2(\R)}\bigr|
	\ge c_pQ^{1+1/[p(p-1)]},
	\]
	where
	\[
	\widetilde{\mathbf f}=\mathcal QP^{\vec N}\mathbf f,
	\qquad
	\widetilde{\mathbf g}=\mathcal QP^{\vec N}\mathbf g.
	\]
	
	Here \(Q\), and hence the weight \(W\) and the pair
	\(\mathbf f,\mathbf g\), are fixed, so the threshold
	\(\ell=\ell(W,\mathbf f,\mathbf g,1)\) of
	Lemma~\ref{lem:norm-tracking} is a fixed finite number. Since the
	sequence of Lemma~\ref{lem:transfer} satisfies
	\(N_k^{(j)}\to\infty\) for every \(k\le r\), we may discard finitely
	many of its terms and thereby assume that its frequencies also
	{
satisfy \(N_k^{(j)}\ge\ell\) for every \(k\le r\), with the common
\(r\) fixed above. The transfer identity \eqref{eq:transfer-identity}
and the norm tracking below therefore hold along the same sequence.
} Using these frequencies,
	construct the pointwise remodeled weight
	\(\widetilde W=\varrho^{\vec N}W\) as in
	Subsection~\ref{sec:periodization-remodeling}. By
	Proposition~\ref{cor:pointwise-distributional-correspondence}, it
	agrees almost everywhere with the distributional limit \(A\). Extend
	it \(1\)-periodically. Then
	\[
	[\widetilde W]_{A_p}
	\le16^p[W]_{A_p}^{\mathrm{dy}}
	\le C_pQ.
	\]
	By Lemma~\ref{lem:norm-tracking} with \(\varepsilon=1\), and since
	\(\|\mathbf f\|_{L^p(W)}=\|\mathbf g\|_{L^q(W^{-q/p})}=1\) by
	Lemma~\ref{lem:pairing-form},
	\[
	\|\widetilde{\mathbf f}\|_{L^p(\widetilde W)}\le2,
	\qquad
	\|\widetilde{\mathbf g}\|_{L^q(\widetilde W^{-q/p})}\le2.
	\]
	Weighted Hölder's inequality
	now gives
	\[
	\frac{\|H\widetilde{\mathbf f}\|_{L^p(\widetilde W)}}
	{\|\widetilde{\mathbf f}\|_{L^p(\widetilde W)}}
	\ge
	\frac{|(H\widetilde{\mathbf f},\widetilde{\mathbf g})|}
	{\|\widetilde{\mathbf f}\|_{L^p(\widetilde W)}
		\|\widetilde{\mathbf g}\|_{L^q(\widetilde W^{-q/p})}}
	\ge c_pQ^{1+1/[p(p-1)]}.
	\]
	The pairing is nonzero, so \(\widetilde{\mathbf f}\ne0\). This
	finishes the proof.
\end{proof}

\noindent\textbf{Statement on AI assistance.}
The authors acknowledge AI‑tool usage in the exploratory stage of this paper.
All mathematical arguments and proofs in the final manuscript were
checked and written by the authors.

\bibliographystyle{amsplain}
\bibliography{JQWX-matrix-weight-2-references}
\end{document}